\documentclass{article}
\usepackage{graphicx}
\usepackage{amsmath}
\usepackage{amsthm}
\usepackage{caption}
\providecommand{\keywords}[1]{\textbf{Keywords} #1}
\newtheorem{remark}{Remark}
\newtheorem{result}{Result}
\newtheorem{theorem}{Theorem}
\newtheorem{observation}{Observation}

\begin{document}

\title{Subgrid multiscale stabilized finite element analysis of fully-coupled unified Stokes-Darcy-Brinkman/Transport model}

\author{B.V. Rathish Kumar, Manisha Chowdhury\thanks{ Email addresses: drbvrk11@gmail.com (B.V.R. Kumar) and  chowdhurymanisha8@gmail.com (M. Chowdhury)  } }
      
\date{Indian Institute of Technology Kanpur \\ Kanpur, Uttar Pradesh, India}

\maketitle
\begin{abstract}
In this study, a stabilized finite element analysis of unified Stokes-Darcy-Brinkman system fully coupled with variable coefficient Advection-Diffusion-Reaction equation(VADR) has been carried out. The viscosity of the fluid, involved in Stokes-Darcy flow, depends on the concentration of the solute, whose transport is described by VADR equation. The algebraic subgrid multiscale approach has been employed to arrive at the stabilized coupled variational formulation. For the time discretization the fully implicit Euler scheme has been used. A detailed derivation of both the apriori and aposteriori estimates for the stabilized subgrid multiscale finite element scheme have been presented. Few numerical experiments have been carried out to verify the credibility of the method.
\end{abstract}

 

\section{Introduction}
Study of transport problem coupled with fluid flow equation has always been an active area of research due to its wide range of applications in effectively modelling various physical phenomena of physiological and environmental importance, such as modelling representing flow of drugs into the blood vessels, contamination of pollutant through rivers into ground water etc. One of the mathematical representative of fluid flow problems, the unified form of Stokes-Darcy-Brinkman equation, models fluid flow in porous media. The unified form of Stokes-Darcy-Brinkman model is used in several mathematical and engineering studies \cite{RefL}- \cite{RefP} for modelling fluid flow through porous media with high porosity. \vspace{1mm}\\
Many numerical methods like stabilized multiscale finite element method \cite{RefH}, MINI and Taylor-Hood finite element and the stabilized $P_1-P_1$ and $P_2-P_2$ methods \cite{RefI}, mixed finite element method \cite{RefJ}, uniformly stable finite element method \cite{RefK}, variational multiscale method, specifically an algebraic subgrid scale (ASGS) approach and the orthogonal subscale stabilization (OSS) method \cite{RefF} have been developed to study Stokes-Darcy-Brinkman model. All of these studies concentrate only upon the fluid flow problems. In \cite{RefA},\cite{RefC} authors have considered models with one way or weak coupling between Stokes-Darcy and Transport model. These studies first solve for velocity field and then solve for concentration field with velocity field as an input data. In this paper strong coupling of Brinkman flow problem with VADR equation has been taken into account in the sense that fluid viscosity depends upon the concentration. In \cite{RefB} authors prove the existence-uniqueness of the weak solution of the variational form of Stokes-Darcy-Brinkman/ADR model. Further under constrained viscosity consideration in \cite{RefG} authors derive apriori error estimates for a stabilized mixed finite element scheme for the Stokes-Darcy-ADR model and present mixed finite element results to a one-way coupling problem. \vspace{1mm}\\
In this current study we consider the strongly coupled unified Stokes-Darcy-Brinkman/VADR model and derive subgrid multiscale stabilized finite element method for it. Algebraic approximation of the subscales that arise from the decomposition of the exact solution field into resolvable scale and unresolvable scale, have been used for finite element scheme stabilization. Stabilization parameters are derived following the approach in \cite{RefE}, \cite{RefF} for ASGS method. Apriori error estimates for the current stabilized ASGS finite element method for the unified strongly coupled Stokes-Darcy-Brinkman/VADR have been derived. Further the aposteriori error estimates following the residual approach have also been carried out. Numerical studies have shown the realization of theoretical order of convergence and the robustness of current stabilized ASGS finite element method for Stokes-Darcy-Brinkman-VADR tightly coupled system.  \vspace{1mm} \\
Organisation of the paper is as follows: Section 2  starts from introducing the model and finishes at Subgrid formulation going through two more subsections presenting weak formulation and semi-discrete formulation. Next section has elaborately described the derivation of apriori and aposteriori error estimations for this subgrid formulation. At last section 4 contains numerical results to verify the numerical performance of the method.      
\section{Model problem}
Let $\Omega \subset R^d$, d=2,3 be an open bounded domain with piecewise smooth boundary $\partial\Omega$. For the sake of simplicity in further calculations, we have considered two dimensional model, but it can be easily extended for three dimensional model. Let us first mention the transient Stokes-Darcy(or Brinkman) flow problem for an incompressible fluid as follows: \\
Find $\textbf{u}$: $\Omega$ $\times$ (0,T) $\rightarrow R^2$ and p: $\Omega \times$ (0,T) $\rightarrow R$ such that,
\begin{equation}
\begin{split}
- \mu(c) \Delta \textbf{u} + \sigma \textbf{u} + \bigtriangledown p & = \textbf{f}_1 \hspace{2mm} in \hspace{2mm} \Omega \times (0,T) \\
\bigtriangledown \cdot \textbf{u} &= f_2 \hspace{2mm} in \hspace{2mm} \Omega \times (0,T) \\
\textbf{u} &= \textbf{0} \hspace{2mm} on \hspace{2mm} \partial\Omega \times (0,T) \\
\textbf{u} &= \textbf{u}_0 \hspace{2mm} at \hspace{2mm} t=0 \\
\end{split}
\end{equation} 
where \textbf{u}= ($u_1,u_2$) is the velocity of the fluid or solvent, p is the pressure, $\mu(c)$ is the viscosity of the fluid depending on concentration c of the dispersing mass of the solute, $\sigma$ is the inverse of permeability, $\textbf{f}_1$ is the body force, $f_2$ is source term and \textbf{$u_0$} is the initial velocity.\vspace{2 mm}\\
This Brinkman flow problem is fully-coupled with the following ADR equation with spatially variable coefficients, which represents the transportation of solute in the same domain $\Omega$ along with homogeneous Dirichlet boundary condition.\\
Find c: $\Omega \times$ (0,T) $\rightarrow R$ such that,
\begin{equation}
\begin{split}
\frac{\partial c}{\partial t}- \bigtriangledown \cdot \tilde{\bigtriangledown} c + \textbf{u} \cdot \bigtriangledown c + \alpha c & = g \hspace{2mm} on \hspace{2mm} \Omega \times (0,T) \\
c &= 0 \hspace{2mm} in \hspace{2mm}\partial \Omega \times (0,T) \\
c & = c_0 \hspace{2mm} at \hspace{2mm} t=0\\
\end{split}
\end{equation}
where the notation, $\tilde{\bigtriangledown}: = (D_1 \frac{\partial}{\partial x}, D_2 \frac{\partial}{\partial y})$ \\
$D_1, D_2$ are variable diffusion coefficients, $\alpha$ is the reaction coefficient and g denotes the source of solute mass. \vspace{1.0 mm}\\ 
Letting \textbf{U}= (\textbf{u},p,c) the equations all together can be written in the following operator form,
\begin{equation}
M\partial_t \textbf{U} + \mathcal{L} \textbf{U} = \textbf{F}
\end{equation}
where M, a matrix = diag(0,0,0,1), $\partial_t \textbf{U} = (\frac{\partial \textbf{u}}{\partial t}, \frac{\partial p}{\partial t}, \frac{\partial c}{\partial t})^T$ \\
\[
\mathcal{L} \textbf{U}=
  \begin{bmatrix}
    - \mu(c) \Delta \textbf{u} + \sigma \textbf{u} + \bigtriangledown p\\
    \bigtriangledown \cdot \textbf{u} \\
    - \bigtriangledown \cdot \tilde{\bigtriangledown} c + \textbf{u} \cdot \bigtriangledown c + \alpha c 
  \end{bmatrix}
\]
 and \[
\textbf{F}=
  \begin{bmatrix}
    \textbf{f}_1 \\
    f_2 \\
    g
  \end{bmatrix}
\]
Let us introduce the adjoint $\mathcal{L}^*$ of $\mathcal{L}$ as follows,
\[
\mathcal{L}^* \textbf{U}=
  \begin{bmatrix}
   - \mu(c) \Delta \textbf{u} + \sigma \textbf{u} - \bigtriangledown p\\
    -\bigtriangledown \cdot \textbf{u} \\
    - \bigtriangledown \cdot \tilde{\bigtriangledown} c - \textbf{u} \cdot \bigtriangledown c + \alpha c 
  \end{bmatrix}
\]
Now we impose suitable assumptions, that are necessary to conclude the results further, on the coefficients mentioned above.\\

\textbf{(i)} The fluid viscosity $\mu(c)= \mu \in C^0(R^+; R^+)$, the space of positive real valued functions defined on positive real numbers and we will have two positive real numbers $\mu_l$ and $\mu_u$ such that 
\begin{equation}
0 < \mu_l \leq \mu(x) \leq \mu_u \hspace{2mm} for \hspace{2mm} any \hspace{2mm} x\in R^+
\end{equation}

\textbf{(ii)} $D_1= D_1(x,y) \in C^0(R^2;R)$ and $D_2= D_2(x,y) \in C^0(R^2;R)$ where $ C^0(R^2;R)$ is the space of real valued continuous function defined on $R^2$. Both are bounded quantity that is we can find lower and upper boundes for both of them. \\

\textbf{(iii)} $\sigma$ and $\alpha$ are positive constants.\\

\textbf{(iv)} The body force, $\textbf{f}_1$ $\in l^{\infty}(0,T; (H^{-1}(\Omega))^2)$ and the source terms $f_2$, g $\in l^{\infty}(0,T; L^2(\Omega))$.

\subsection{Weak formulation} 
Let us first introduce the spaces as follows,\\
$H^1(\Omega)= \{v \in L^2(\Omega) : \bigtriangledown v \in L^2(\Omega) \} $ \vspace{1mm} \\
Let $V_s=H^1_0(\Omega) = \{ v \in H^1 (\Omega): v=0 \hspace{1 mm} on \hspace{1 mm} \partial \Omega \}$ and $Q_s=L^2(\Omega)$ and J= (0,T)\vspace{1mm}\\
Let $\textbf{V} $ := $l^2(0,T; V_s)\bigcap l^{\infty}(0,T; Q_s)$ and $\textbf{Q} $ := $l^2(0,T; Q_s)$ \vspace{1 mm}\\
Let us introduce another notation $\textbf{V}_F = V_s \times V_s\times Q_s \times V_s$ \vspace{2mm}\\
The weak formulation of (1)-(2) is to find \textbf{U}= (\textbf{u},p,c): J $ \rightarrow \textbf{V}_F$ such that $\forall$ \textbf{V}=(\textbf{v},q,d) $\in  \textbf{V}_F$
\begin{equation}
 (\frac{\partial c}{\partial t}, d) + a_S(\textbf{u},\textbf{v})- b(\textbf{v},p)+ b(\textbf{u},q)+ a_T(c,d)= l_S^1(\textbf{v})+ l_S^2(q)+l_T(d) 
\end{equation} 
where $a_S(\textbf{u},\textbf{v})= \int_{\Omega} \mu(c) \bigtriangledown \textbf{u}:\bigtriangledown \textbf{v} + \sigma \int_{\Omega} \textbf{u} \cdot \textbf{v}$ \vspace{1 mm}\\
 $b(\textbf{v},q)= \int_{\Omega} (\bigtriangledown \cdot \textbf{v}) q$ \vspace{1 mm} \\
 $a_T(c,d) = \int_{\Omega} \tilde{\bigtriangledown}c \cdot \bigtriangledown d + \int_{\Omega} d \textbf{u} \cdot \bigtriangledown c + \alpha\int_{\Omega}cd $ \vspace{1 mm} \\
 $l_S^1 (\textbf{v})= \int_{\Omega} \textbf{f} \cdot \textbf{v}$, $l_S^2(q)=\int_{\Omega} f_2 q$ and  $l_T(d)= \int_{\Omega} gd$ \vspace{1 mm} \\
Again the above formulation can be written as,\\
Find $\textbf{U} \in \textbf{V}_F$ such that
\begin{equation}
(M\partial_t \textbf{U},\textbf{V}) + B(\textbf{U}, \textbf{V}) = L(\textbf{V})   \hspace{2 mm} \forall \textbf{V} \in \textbf{V}_F
\end{equation}
where $B(\textbf{U}, \textbf{V}) = a_S(\textbf{u},\textbf{v})- b(\textbf{v},p)+ b(\textbf{u},q)+ a_T(c,d)$ \vspace{1mm}\\
 $L(\textbf{V})= l_S^1(\textbf{v})+l_S^2(q)+ l_T(d) $ \vspace{2 mm}\\
\begin{remark}
\cite{RefF} discusses about well posedness of unified Stokes-Darcy equation for positive viscosity coefficient.
\end{remark}

\begin{remark}
The existence of the weak solution of the variational form for coupled Stokes-Darcy/transport equation has been discussed in \cite{RefB}. Under the assumptions [(i)-(iv)] the existence of unique weak solution of the variational form (6) can be established easily following the approach presented in \cite{RefB}, as this model contains only linear terms.
\end{remark}

\subsection{Semi-discrete formulation}
In this section we will introduce the standard Galerkin finite element space discretization for the above variational form (5).\vspace{1 mm} \\
Let the domain $\Omega$ be discretized into finite numbers of subdomains $\Omega_k$ for k=1,2,...,$n_{el}$, where $n_{el}$ is the total number element subdomains. Let $h_k$ be the diameter of each subdomain $\Omega_k$ and h= $\underset{k=1,2,...n_{el}}{max} h_k$ \vspace{1 mm}\\
Let $\tilde{\Omega}= \bigcup_{k=1}^{n_el} \Omega_k$ be the union of interior elements.\vspace{1 mm}\\
Let $V_s^h$ and $Q_s^h$ be finite dimensional subspaces of $V_s$ and $Q_s$ respectively. They are taken as follows, \vspace{1 mm} \\
$V_s^h= \{ v \in V_s: v(\Omega_k)= \mathcal{P}^2(\Omega_k)\} $ \vspace{1 mm} \\
$Q_s^h= \{ q \in Q_s : q(\Omega_k)= \mathcal{P}^1(\Omega_k)\}$ \vspace{1 mm}\\
where $\mathcal{P}^1(\Omega_k)$ and $\mathcal{P}^2(\Omega_k)$ denote complete polynomial of order 1 and 2 respectively over each $\Omega_k$ for k=1,2,...,$n_{el}$. \\
Let us consider similar notation $\textbf{V}_F^h$ for corresponding finite dimensional subspace of $\textbf{V}_F$ where $\textbf{V}_F^h=V_s^h \times V_s^h \times Q_s^h \times V_s^h $\vspace{2 mm} \\
Now the Galerkin formulation of the variational form (6) will be as follows:\\
Find $\textbf{U}_h $= $(\textbf{u}_h,p_h,c_h)$: J $ \rightarrow \textbf{V}_F^h$ such that $\forall$ $\textbf{V}_h=(\textbf{v}_h,q_h,d_h)$ $\in \textbf{V}_F^h$
\begin{equation}
(M\partial_t \textbf{U}_h,\textbf{V}_h) + B(\textbf{U}_h, \textbf{V}_h) = L(\textbf{V}_h)   
\end{equation}
where $(M\partial_t \textbf{U}_h,\textbf{V}_h)$= $ (\frac{\partial c_h}{\partial t}, d_h)$ \vspace{1mm}\\
 $B(\textbf{U}_h, \textbf{V}_h) = a_S(\textbf{u}_h,\textbf{v}_h)- b(\textbf{v}_h,p_h)+ b(\textbf{u}_h,q_h)+ a_T(c_h,d_h)$ \vspace{1mm}\\
 $L(\textbf{V}_h)= l_S^1(\textbf{v}_h)+l_S^2(q_h) + l_T(d_h) $                               
 
  \subsection{Subgrid multiscale formulation}
This stabilization method has been introduced to correct the lack of stability that the Galerkin method suffers due to small diffusion coefficient. It involves decomposition of the solution space $\textbf{V}_F$ into the spaces of resolved scales and unresolved scales. The finite element space $\textbf{V}_F^h$ is considered as the space of resolved scales. Then the final form of subgrid formulation will be arrived while the elements of unresolved scales will be expressed in the terms of elements of resolved scales. \vspace{1 mm} \\
Following the procedure described in \cite{RefD} the variational subgrid scale model for this coupled equation will be written as follows,\vspace{1 mm} \\
Find $\textbf{U}_h $= $(\textbf{u}_h,p_h,c_h)$: J $ \rightarrow \textbf{V}_F^h$ such that $\forall$ $\textbf{V}_h=(\textbf{v}_h,q_h,d_h)$ $\in \textbf{V}_F^h$ 
\begin{equation}
(M\partial_t \textbf{U}_h,\textbf{V}_h) + B_{ASGS}(\textbf{U}_h, \textbf{V}_h)  = L_{ASGS}(\textbf{V}_h)  
\end{equation}
where $B_{ASGS}(\textbf{U}_h, \textbf{V}_h)= B(\textbf{U}_h, \textbf{V}_h)+ \sum_{k=1}^{n_{el}} (\tau_k'(M\partial_t \textbf{U}_h + \mathcal{L}\textbf{U}_h-\textbf{d}), -\mathcal{L}^*\textbf{V}_h)_{\Omega_k}- \sum_{k=1}^{n_{el}}((I-\tau_k^{-1}\tau_k')(M\partial_t \textbf{U}_h + \mathcal{L}\textbf{U}_h), \textbf{V}_h)_{\Omega_k}-\sum_{k=1}^{n_{el}} (\tau_k^{-1}\tau_k' \textbf{d}, \textbf{V}_h)_{\Omega_k}$ \vspace{2 mm}\\
$L_{ASGS}(\textbf{V}_h)= L(\textbf{V}_h)+ \sum_{k=1}^{n_{el}}(\tau_k' \textbf{F}, -\mathcal{L}^*\textbf{V}_h)_{\Omega_k}- \sum_{k=1}^{n_{el}}((I-\tau_k^{-1}\tau_k')\textbf{F}, \textbf{V}_h)_{\Omega_k}$  \vspace{1 mm} \\
where the stabilization parameter $\tau_k$ is in matrix form as 
\[
\tau_k= diag(\tau_{1k},\tau_{1k},\tau_{2k},\tau_{3k}) =
  \begin{bmatrix}
    \tau_{1k} I & 0 & 0 \\
    0 & \tau_{2k} & 0 \\
    0 & 0 & \tau_{3k}
  \end{bmatrix}
\]
and 
\[
\tau_k'= (\frac{1}{dt}M+ \tau_k^{-1})^{-1} =
  \begin{bmatrix}
    \tau_{1k}I & 0 & 0 \\
    0 & \tau_{2k} & 0 \\
    0 & 0 & \frac{\tau_{3k} dt}{dt+ \tau_{3k}}
  \end{bmatrix}\\
  = diag (\tau_{1k}',\tau_{1k}',\tau_{2k}',\tau_{3k}') \hspace{2mm}(say)
\]
I is a square identity matrix of order 2.\vspace{1 mm}\\
$\textbf{d}$= $\sum_{i=1}^{n+1}(\frac{1}{dt}M\tau_k')^i(\textbf{F} -M\partial_t \textbf{U}_h - \mathcal{L}\textbf{U}_h)$\vspace{1 mm}\\
considering $d_i$ for i=1,2,3,4 are components of the matrix $\textbf{d}$ and it can be easily observed that $d_1,d_2,d_3$ are always 0 because of matrix M. \vspace{2 mm}\\
We have the forms of the stabilization parameters $\tau_{1k}, \tau_{2k}$ for unified Stokes-Darcy problem in \cite{RefF} and $\tau_{3k}$ for ADR equation with spatially variable coefficients in \cite{RefE} and for each k all the coefficients $\tau_{ik}$ coincide with $\tau_{i}$ for i=1,2,3 that is, for each k=1,2,...,$n_{el}$
\begin{equation}
\begin{split}
\tau_{1k} &= \tau_{1}= (c_1^{\textbf{u}} \frac{\mu_l}{h^2}+  c_2^{\textbf{u}} \sigma)^{-1} \\
\tau_{2k} &=\tau_{2}=(c_1^p \mu_l+ c_2^p \sigma h^2) \\
\tau_{3k} & = \tau_{3}= (\frac{9D}{4h^2} + \frac{3U}{2h} + \alpha )^{-1}
\end{split}
\end{equation}

\section{Error estimates}
We start this section with the introduction of the notion of error terms, followed by splitting of those error terms through introducing the projection operator corresponding to each unknown variable. Later we have introduced fully-discrete formulation and then conducted  apriori and aposteriori error estimates.

\subsection{Projection operators : Error splitting}
Let $\textbf{e}=(e_{\textbf{u}},e_p,e_c)$ denote the error where the components are $e_{\textbf{u}}=(e_{u1},e_{u2})= (u_1-u_{1h}, u_2-u_{2h}), e_p= (p-p_h)$ and $e_c=(c-c_h)$.
Here $\textbf{u}=(u_1,u_2)$ and all the remaining notations carry their respective meanings. \vspace{2mm}\\
Let us introduce the projection operator for each of this error components.\vspace{1 mm}\\
(i)For any $\textbf{u} \in V_s \times V_s $ we assume that there exists an interpolation $I^h_{\textbf{u}}:  V_s \times V_s \longrightarrow  V_s^h \times V_s^h $ satisfying \vspace{1mm}\\
(a) $b(\textbf{u}-I^h_{\textbf{u}}\textbf{u}, q_h)=0$ \hspace{2mm} $\forall q_h \in Q_s^h$ and \vspace{1mm}\\
each component of the projection map that is $I^h_{u_1}: V_s \longrightarrow V_s^h$ and $I^h_{u_2}: V_s \longrightarrow V_s^h$ are $L^2$ orthogonal projections, satisfying \vspace{1mm}\\
(b) for any $u_1 \in V_s$ \hspace{1mm} $(u_1-I^h_{u_1}u_1,v_{1h})=0$ \hspace{1mm} $\forall v_{1h} \in V_s^h$ and \vspace{1mm} \\
(c)for any $u_2 \in V_s$ \hspace{1mm} $(u_2-I^h_{u_2} u_2,v_{2h})=0$ \hspace{1mm} $\forall v_{2h} \in V_s^h$ \vspace{2mm}\\
(ii) Let $I^h_p: Q_s \longrightarrow Q_s^h$ be the $L^2$ orthogonal projection given by \\ $\int_{\Omega}(p-I^h_pp)q_h=0$ \hspace{1mm} $\forall p \in Q_s, \hspace{1mm} \forall q_h \in Q_s^h$ \vspace{2mm}\\
(iii)Let $I^h_{c}: V_s \longrightarrow V_s^h$ be the $L^2$ orthogonal projection given by \\ $\int_{\Omega}(c-I^h_c c)d_h=0$ \hspace{1mm} $\forall c \in V_s, \hspace{1mm} \forall d_h \in V_s^h$ \vspace{2mm}\\
Now each components of the error can be split into two parts interpolation part, $E^I$ and auxiliary part, $E^A$ as follows: \vspace{1mm}\\
$e_{u1}=(u_1-u_{1h})=(u_1-I^h_{u1}u_1)+(I^h_{u1}u_1-u_{1h})= E^{I}_{u1}+ E^{A}_{u1}$ \vspace{1mm}\\
Similarly, $e_{u2}=E^{I}_{u2}+ E^{A}_{u2}$,
$e_{p}=E^{I}_{p}+ E^{A}_{p}$, and 
$e_{c}=E^{I}_{c}+ E^{A}_{c}$ \vspace{2mm}\\
Now we put some results using the properties of projection operators and these results will be used in error estimations.
\begin{result}
\begin{equation}
(\frac{\partial}{\partial t} E^I_{c}, d_h)=0 \hspace{2mm} d_h \in V_s^h
\end{equation}
\end{result}
\textit{Proof:} We have $(c-I^h_{c}c,d_h)=0= (E^I_{c},d_h)$ \hspace{1mm} $\forall d_h \in V_s^h$ \vspace{1mm}\\
Therefore
\begin{equation}
\begin{split}
\frac{d}{dt}(E^I_{c},d_h) & =0 \hspace{1mm} \forall d_h \in V_s^h \\
(\frac{\partial}{\partial t} E^I_{c},d_h) + (E^I_{c}, \frac{\partial}{\partial t} d_h) & =0 \hspace{1mm} \forall d_h \in V_s^h \\
(\frac{\partial}{\partial t} E^I_{c},d_h) &=0 \hspace{1mm} \forall d_h \in V_s^h \\
\end{split}
\end{equation}
Since $\frac{\partial}{\partial t} d_h \in V_s^h$, the second term in second equation $(E^I_{c}, \frac{\partial}{\partial t} d_h)=0$ \vspace{2mm}\\
Useful interpolation estimation results [ref] are as follows: for any exact solution with regularity upto (m+1)
\begin{equation}
\|v-I^h_v v\|_l = \|E^I_v\|_l \leq C(p,\Omega) h^{m+1-l} \|v\|_{m+1} 
\end{equation}
where l ($\leq m+1$) is a positive integer and C is a constant depending on m and the domain. For l=0 and 1 it  implies standard $L^2(\Omega)$ and $H^1(\Omega)$ norms respectively. For simplicity we will use $\| \cdot \|$ instead of $\| \cdot \|_0$ to denote $L^2(\Omega)$ norm.

\subsection{Fully-discrete form}
Before introducing time discretization, some notations  have been introduced: for $dt$= $\frac{T}{N}$, where $N$ is a positive integer, $t_n= n dt$ and for given $0 \leq \theta \leq 1$,
\begin{equation}
\begin{split}
f^n & = f(\cdot , t_n) \hspace{4 mm} for \hspace{2 mm} 0 \leq n \leq N\\
f^{n,\theta} &= \frac{1}{2} (1 + \theta) f^{(n+1)} + \frac{1}{2} (1- \theta) f^n \hspace{4mm} for \hspace{2mm} 0\leq n \leq N-1
\end{split}
\end{equation}
Later we will see for $\theta=0$ the discretization follows Crank-Nicolson formula and for $\theta=1$ it is backward Euler discretization rule.\vspace{1mm}\\
For sufficiently smooth function f(t), using the Taylor series expansion about t= $t^{n,\theta}$, we will have \vspace{1mm}\\
\begin{equation}
\begin{split}
f^{n+1} & = f(t^{n,\theta})+ \frac{(1-\theta)  dt}{2} \frac{\partial f}{\partial t}(t^{n,\theta}) + \frac{(1-\theta)^2 dt^2}{8} \frac{\partial^2 f}{\partial t^2} (t^{n,\theta}) + \mathcal{O}(dt^3)\\
f^{n} & = f(t^{n,\theta})- \frac{(1+\theta) dt}{2} \frac{\partial f}{\partial t}(t^{n,\theta}) + \frac{(1+\theta)^2 dt^2}{8} \frac{\partial^2 f}{\partial t^2}(t^{n,\theta}) + \mathcal{O}(dt^3)
\end{split}
\end{equation}
We have considered here $t^{n,\theta}- t^n= \frac{(1+\theta) \Delta t}{2}$\\
Multiplying the first and second equations by $\frac{1+\theta}{2}$ and $\frac{1-\theta}{2}$ respectively and then adding them we will have the following according to (16)\\
\begin{equation}
f^{n,\theta} = f(t^{n,\theta}) + \frac{1}{8} (1+\theta)(1-\theta)\Delta t^2 \frac{\partial^ f}{\partial t^2}(t^{n,\theta}) + \mathcal{O}(dt^3)
\end{equation} 
For $\theta=1$ we will have third order accuracy in time and for $\theta=0$ the scheme will be second order accurate in time.\vspace{1mm}\\
Let $\textbf{u}^{n,\theta},p^{n,\theta},c^{n,\theta}$ be approximations of $\textbf{u}(\textbf{x},t^{n,\theta}), p(\textbf{x},t^{n,\theta}),c(\textbf{x},t^{n,\theta})$ respectively. Now by Taylor series expansion \cite{RefQ},we have 
\begin{equation}
\begin{split}
\frac{c^{n+1}-c^n}{dt} & = c_t(\textbf{x},t^{n,\theta}) + TE \hspace{5mm} \forall \textbf{x} \in \Omega
\end{split}
\end{equation}
where the truncation error $TE$ depends upon time-derivatives of the respective variables and $dt$.
\begin{equation}
\begin{split}
\|TE^2\| & \leq
      \begin{cases}
      C' dt \|c_{tt}\|_{L^{\infty}(t^n,t^{n+1},H^1)} & if \hspace{1mm} \theta=1 \\
    C" dt^2 \|c_{ttt}\|_{L^{\infty}(t^n,t^{n+1},H^1)} & if \hspace{1mm} \theta=0
      \end{cases}
\end{split}
\end{equation}
The truncation error is of $O(dt \theta +dt^2(1-\theta)^3+dt^2(1+\theta)^3)$ \cite{RefQ}\\
After introducing all the required definitions finally the fully-discrete formulation of sub-grid form is as follows: \\
For given $\textbf{U}_h^n = (\textbf{u}_h^n,p_h^n,c_h^n)\in \textbf{V}_F^h$ find $\textbf{U}_h^{n+1}= (\textbf{u}_h^{n+1},p_h^{n+1},c_h^{n+1}) \in \textbf{V}_F^h $  such that , $\forall \hspace{1mm} \textbf{V}_h=(\textbf{v}_h,q_h,d_h) \in \textbf{V}_F^h $
\begin{equation}
(M\frac{(\textbf{U}_h^{n+1}-\textbf{U}_h^n)}{dt}, \textbf{V}_h)+ B_{ASGS}(\textbf{U}_h^{n,\theta}, \textbf{V}_h) = L_{ASGS}(\textbf{V}_h) + (TE,d_h)  
\end{equation}
Again for the exact solution we will have the discrete formulation as follows: \\
For given $\textbf{U}^n = (\textbf{u}^n,p^n,c^n)\in \textbf{V}_F$ find $\textbf{U}^{n+1}= (\textbf{u}^{n+1},p^{n+1},c^{n+1}) \in \textbf{V}_F $  such that , $\forall \hspace{1mm} \textbf{V}_h=(\textbf{v}_h,q_h,d_h) \in \textbf{V}_F^h$
\begin{equation}
(M\frac{(\textbf{U}^{n+1}-\textbf{U}^n)}{dt}, \textbf{V}_h)+ B(\textbf{U}^{n,\theta}, \textbf{V}_h) = L(\textbf{V}_h) + (TE,d_h)  
\end{equation}

\subsection{Apriori error estimation}
In this section we will find apriori error bound, which depends on the exact solution. Here we first estimate auxiliary error bound and later using that we will find apriori error estimate. Before deriving error estimations let us mention few definitions of norm which we are going to use in this estimation:
\begin{equation}
\begin{split}
\|f\|_{l^2(H^1)}^2 &= \sum_{n=0}^{N-1} (\int_{\Omega} \mid f^{n,\theta} \mid^2  +  \int_{\Omega} \mid \frac{\partial f}{\partial x}^{n,\theta} \mid^2  +  \int_{\Omega} \mid \frac{\partial f}{\partial y}^{n,\theta} \mid^2 )dt \\
\|f\|_V^2 & = \underset{0\leq n \leq N}{max} \|f^n\|^2 + \|f\|_{l^2(H^1)}^2\\
\|f\|_Q^2 &= \|f\|_{l^2(L^2)}^2 = \sum_{n=0}^{N-1} \|f^{n,\theta}\|^2 dt 
\end{split}
\end{equation}
\begin{theorem} (Auxiliary error estimate)
For velocity $\textbf{u}_h=(u_{1h},u_{2h})$, pressure $p_h$ and concentration $c_h$ belonging to $V_s^h \times V_s^h \times Q_s^h \times V_s^h$ satisfying (10), assume dt is sufficiently small and positive, and sufficient regularity of exact solution in equations (1)-(2). Then there exists a constant C, depending upon $\textbf{u}$,p,c , such that
\begin{equation}
\|E^A_{u1}\|^2_{l^2(H^1)} + \|E^A_{u2}\|^2_{l^2(H^1)}+ \|E^A_p\|_{l^2(L^2)}^2  + \|E^A_{c}\|^2_V \leq C (h^2+h+ dt^{2r})
\end{equation}
where
\begin{equation}
    r=
    \begin{cases}
      1, & \text{if}\ \theta=1 \\
      2, & \text{if}\ \theta=0
    \end{cases}
  \end{equation}
\end{theorem}
\begin{proof} In first part we will find bound for auxiliary error part of velocity $\textbf{u}$ and concentration c with respect to $V$ norm and in the second part we will estimate auxiliary error for pressure term with respect to $Q$ norm and finally combining them we will arrive at the desired result. \vspace{2mm} \\
\textbf{First part} Subtracting (18) from (19) in combined form and then simplifying the terms, we have $\forall \hspace{1mm} \textbf{V}_h \in V_s^h \times V_s^h \times Q_s^h \times V_s^h$

\begin{multline}
(M\frac{(\textbf{U}^{n+1}-\textbf{U}^{n+1}_{h})- (\textbf{U}^{n}-\textbf{U}^{n}_{h})}{dt},V_{h}) + B(\textbf{U}^{n,\theta}-\textbf{U}^{n,\theta}_h, \textbf{V}_h)\\
+ \sum_{k=1}^{n_{el}}\tau_k'(M\partial_t (\textbf{U}^n-\textbf{U}^n_h)+ \mathcal{L}(\textbf{U}^{n,\theta}-\textbf{U}_h^{n,\theta}),-\mathcal{L}^* \textbf{V}_h)_{\Omega_k} -\sum_{k=1}^{n_{el}}\tau_k' (\textbf{d},-\mathcal{L}^* \textbf{V}_h)_{\Omega_k}\\
+\sum_{k=1}^{n_{el}}((I-\tau_k^{-1}\tau_k')(M \partial_t(\textbf{U}^{n}-\textbf{U}^{n}_h) + \mathcal{L}(\textbf{U}^{n,\theta}-\textbf{U}^{n,\theta}_h)), -\textbf{V}_h)_{\Omega_k}+ \sum_{k=1}^{n_{el}} (\tau_k^{-1}\tau_k' \textbf{d}, \textbf{V}_h)_{\Omega_k}=0
\end{multline}
where $\textbf{d}$= $(\sum_{i=1}^{n+1}(\frac{1}{dt}M\tau_k')^i)(M\partial_t (\textbf{U}^n-\textbf{U}^n_h) + \mathcal{L}(\textbf{U}^{n,\theta}-\textbf{U}_h^{n,\theta}))$ \vspace{2mm}\\
Let us divide the big expressions into small parts, then using error splitting in each of them and simplifying further, we will have them as follows: \vspace{1mm}\\
Let
\begin{equation}
\begin{split}
I_1 &= (M\frac{(\textbf{U}^{n+1}-\textbf{U}^{n+1}_{h})- (\textbf{U}^{n}-\textbf{U}^{n}_{h})}{dt},V_{h}) \\
& = (\frac{(c^{n+1}-c^{n+1}_{h})- (c^{n}-c^{n}_{h})}{dt},d_{h})\\
 & =  (\frac{(E^{I,n+1}_{c}+E^{A,n+1}_{c})- (E^{I,n}_{c}+E^{A,n}_{c})}{dt},d_{h})\\
 & =  (\frac{E^{A,n+1}_{c}-E^{A,n}_{c}}{dt},d_{h})
\end{split}
\end{equation}
We arrive at the last line after using result 1, deduced in the previous section. \\
\begin{equation}
\begin{split}
I_2 & = B(\textbf{U}^{n,\theta}-\textbf{U}^{n,\theta}_h, \textbf{V}_h) \\
& = \int_{\Omega} \mu(c^n)\bigtriangledown (\textbf{u}^{n,\theta}-\textbf{u}^{n,\theta}_h): \bigtriangledown \textbf{v}_h + \int_{\Omega} \sigma (u_1^{n,\theta}-u^{n,\theta}_{1h})v_{1h}\\
& \quad + \int_{\Omega} \sigma (u_1^{n,\theta}-u^{n,\theta}_{1h})v_{2h}-\int_{\Omega}(\bigtriangledown \cdot \textbf{v}_h)(p^{n,\theta}-p_h^{n,\theta})+ \int_{\Omega}(\bigtriangledown \cdot \textbf{u}^{n,\theta}-\textbf{u}_h^{n,\theta})q_h \\
& \quad + \int_{\Omega} \tilde{\bigtriangledown}(c^{n,\theta}-c^{n,\theta}_h) \cdot \bigtriangledown d_h + \int_{\Omega}d_h \textbf{u}^n \cdot \bigtriangledown (c^{n,\theta}-c^{n,\theta}_h) + \int_{\Omega} \alpha (c^{n,\theta}-c^{n,\theta}_h) d_h \\
& = \int_{\Omega} \mu(c^n)\bigtriangledown (E^{I,n,\theta}_{\textbf{u}} + E^{A,n,\theta}_{\textbf{u}} ): \bigtriangledown \textbf{v}_h + \int_{\Omega} \sigma (E^{I,n,\theta}_{u1} + E^{A,n,\theta}_{u1})v_{1h}\\
& \quad + \int_{\Omega} \sigma (E^{I,n,\theta}_{u2} + E^{A,n,\theta}_{u2})v_{2h}-\int_{\Omega}(\bigtriangledown \cdot \textbf{v}_h)(E^{I,n,\theta}_{p} + E^{A,n,\theta}_{p})\\
& \quad + \int_{\Omega}\bigtriangledown \cdot (E^{I,n,\theta}_{\textbf{u}} + E^{A,n,\theta}_{\textbf{u}})q_h + \int_{\Omega} \tilde{\bigtriangledown}(E^{I,n,\theta}_{c} + E^{A,n,\theta}_{c}) \cdot \bigtriangledown d_h\\
& \quad + \int_{\Omega}d_h \textbf{u}^n \cdot \bigtriangledown (E^{I,n,\theta}_{c} + E^{A,n,\theta}_{c}) + \int_{\Omega} \alpha (E^{I,n,\theta}_{c} + E^{A,n,\theta}_{c}) d_h \\
\end{split}
\end{equation}
\begin{equation}
\begin{split}
& = \int_{\Omega} \mu(c^n)\bigtriangledown E^{I,n,\theta}_{\textbf{u}} : \bigtriangledown \textbf{v}_h + \int_{\Omega} \mu(c^n)\bigtriangledown E^{A,n,\theta}_{\textbf{u}} : \bigtriangledown \textbf{v}_h \\
& \quad + \int_{\Omega} \sigma E^{A,n,\theta}_{u1}v_{1h} + \int_{\Omega} \sigma E^{A,n,\theta}_{u2}v_{2h}- \int_{\Omega}(\bigtriangledown \cdot \textbf{v}_h)(E^{I,n,\theta}_{p} + E^{A,n,\theta}_{p}) \\
& \quad + \int_{\Omega}(\bigtriangledown \cdot  E^{A,n,\theta}_{\textbf{u}}) q_h + \int_{\Omega} \tilde{\bigtriangledown}E^{I,n,\theta}_{c} \cdot \bigtriangledown d_h + \int_{\Omega} \tilde{\bigtriangledown} E^{A,n,\theta}_{c} \cdot \bigtriangledown d_h \\
& \quad + \int_{\Omega}d_h \textbf{u}^n \cdot \bigtriangledown E^{I,n,\theta}_{c} + \int_{\Omega}d_h \textbf{u}^n \cdot \bigtriangledown E^{A,n,\theta}_{c} + \int_{\Omega} \alpha E^{A,n,\theta}_{c} d_h
\end{split}
\end{equation}
Above we have used various properties of the projection operators and arrived at the last expression.
\begin{equation}
\begin{split}
I_3 &= \sum_{k=1}^{n_{el}}(\tau_k'(M\partial_t (\textbf{U}^n-\textbf{U}^n_h)+ \mathcal{L}(\textbf{U}^{n,\theta}-\textbf{U}_h^{n,\theta})-\textbf{d}),-\mathcal{L}^* \textbf{V}_h)_{\Omega_k} \\
&= \sum_{k=1}^{n_{el}} \{(\tau_1'( - \mu(c) \Delta (u_1^{n,\theta}-u_{1h}^{n,\theta}) + \sigma (u_1^{n,\theta}-u_{1h}^{n,\theta}) + \frac{\partial (p^{n,\theta}-p_{h}^{n,\theta})}{\partial x}-d_1),(\mu(c) \Delta v_{1h} \\
& \quad  - \sigma (v_{1h}) + \frac{\partial q_h}{\partial x}))_{\Omega_k} +(\tau_1'( - \mu(c) \Delta (u_2^{n,\theta}-u_{2h}^{n,\theta}) + \sigma (u_2^{n,\theta}-u_{2h}^{n,\theta}) + \frac{\partial (p^{n,\theta}-p_{h}^{n,\theta})}{\partial y}\\ 
&\quad -d_2),(\mu(c) \Delta v_{2h} - \sigma v_{2h} + \frac{\partial q_h}{\partial y}))_{\Omega_k} + (\tau_2'(\bigtriangledown \cdot (\textbf{u}^{n,\theta}-\textbf{u}_h^{n,\theta})-d_3), \bigtriangledown \cdot \textbf{v}_h)_{\Omega_k} + \\
 & \quad  (\tau_3'( \partial_t(c^n-c_h^n) - \bigtriangledown \cdot \tilde{\bigtriangledown}(c^{n,\theta}-c_h^{n,\theta})+ \textbf{u}^n \cdot \bigtriangledown (c^{n,\theta}-c_h^{n,\theta}) + \alpha (c^{n,\theta}-c_h^{n,\theta})-d_4),\\ 
& \quad \bigtriangledown \cdot \tilde{\bigtriangledown} d_h +\textbf{u}^n \cdot \bigtriangledown d_h - \alpha d_h )_{\Omega_k}\} \\
& = \sum_{k=1}^{n_{el}} \{(\tau_1'(- \mu(c) \Delta (E^{I,n,\theta}_{u1}+E_{u1}^{A,n,\theta}) + \sigma (E^{I,n,\theta}_{u1}+E_{u1}^{A,n,\theta}) + \frac{\partial( E^{I,n,\theta}_{p})}{\partial x} +\frac{\partial( E^{A,n,\theta}_{p})}{\partial x} ), \\
& \quad (\mu(c) \Delta v_{1h} - \sigma v_{1h} + \frac{\partial q_h}{\partial x}))_{\Omega_k} +(\tau_1'(- \mu(c) \Delta (E^{I,n,\theta}_{u2}+ E_{u2}^{A,n,\theta}) +  \sigma (E^{I,n,\theta}_{u2}+E_{u2}^{A,n,\theta}) + \\
&\quad  \frac{\partial ( E^{I,n,\theta}_{p}+E_{p}^{A,n,\theta})}{\partial y}),
 (\mu(c) \Delta v_{2h} - \sigma v_{2h} + \frac{\partial q_h}{\partial y}))_{\Omega_k} + (\tau_2'\bigtriangledown \cdot (E^{I,n,\theta}_{\textbf{u}}+E_{\textbf{u}}^{A,n,\theta}), \bigtriangledown \cdot \textbf{v}_h)_{\Omega_k}\\ 
 & \quad + (\tau_3'( \partial_t(E^{I,n}_{c}+E_{c}^{A,n}) - \bigtriangledown \cdot \tilde{\bigtriangledown}(E^{I,n,\theta}_{c}+  E_{c}^{A,n,\theta}) + \textbf{u}^n \cdot \bigtriangledown (E^{I,n,\theta}_{c}+E_{c}^{A,n,\theta}) + \\
& \quad \alpha (E^{I,n,\theta}_{c}+E_{c}^{A,n,\theta})-d_4),\bigtriangledown \cdot \tilde{\bigtriangledown} d_h +\textbf{u}^n \cdot \bigtriangledown d_h - \alpha d_h )_{\Omega_k}\}
\end{split} 
\end{equation}
\begin{equation}
\begin{split}
I_4 & = \sum_{k=1}^{n_{el}}\tau_k'(-\textbf{d},-\mathcal{L}^* \textbf{V}_h)_{\Omega_k} = \sum_{k=1}^{n_{el}}\{ \tau_3' ( d_4, \bigtriangledown \cdot \tilde{\bigtriangledown} d_h +\textbf{u}^n \cdot \bigtriangledown d_h - \alpha d_h )_{\Omega_k}\}
\end{split}
\end{equation}
\begin{equation}
\begin{split}
I_5 &= \sum_{k=1}^{n_{el}}((I-\tau_k^{-1}\tau_k')(M \partial_t(\textbf{U}^{n}-\textbf{U}^{n}_h) + \mathcal{L}(\textbf{U}^{n,\theta}-\textbf{U}^{n,\theta}_h)), -\textbf{V}_h)_{\Omega_k}\\
& = \sum_{k=1}^{n_{el}} \{((1-\tau_1^{-1}\tau_1')( - \mu(c) \Delta (u_1^{n,\theta}-u_{1h}^{n,\theta}) + \sigma (u_1^{n,\theta}-u_{1h}^{n,\theta}) + \frac{\partial (p^{n,\theta}-p_{h}^{n,\theta})}{\partial x}),-v_{1h})_{\Omega_k} \\
& \quad + ((1-\tau_1^{-1}\tau_1')( - \mu(c) \Delta (u_2^{n,\theta}-u_{2h}^{n,\theta}) + \sigma (u_2^{n,\theta}-u_{2h}^{n,\theta}) +\frac{\partial (p^{n,\theta}-p_{h}^{n,\theta})}{\partial y}),-v_{2h})_{\Omega_k} + \\
 &\quad  ((1-\tau_2^{-1}\tau_2')\bigtriangledown \cdot (\textbf{u}^{n,\theta}-\textbf{u}_h^{n,\theta}), -q_h)_{\Omega_k} + ((1-\tau_3^{-1}\tau_3')( \partial_t(c^n-c_h^n) - \bigtriangledown \cdot \tilde{\bigtriangledown}(c^{n,\theta}-c_h^{n,\theta}) \\ 
 & \quad + \textbf{u}^n \cdot \bigtriangledown (c^{n,\theta}-c_h^{n,\theta}) + \alpha (c^{n,\theta}-c_h^{n,\theta})),-d_h )_{\Omega_k}\} \\
 & = \sum_{k=1}^{n_{el}}((1-\tau_3^{-1}\tau_3') ( \partial_t E^{I,n}_{c} + \partial_t E^{A,n}_{c} - \bigtriangledown \cdot \tilde{\bigtriangledown}(E^{I,n,\theta}_{c}+E_{c}^{A,n,\theta})+ \textbf{u} \cdot \bigtriangledown (E^{I,n,\theta}_{c}+E_{c}^{A,n,\theta}) + \\
 & \quad \alpha (E^{I,n,\theta}_{c}+E_{c}^{A,n,\theta})),-d_h )_{\Omega_k}
\end{split}
\end{equation}
since $(1- \tau_1^{-1}\tau_1')=0=(1- \tau_2^{-1}\tau_2')$ and the last term,
\begin{equation}
\begin{split}
I_6 & = \sum_{k=1}^{n_{el}} (\tau_k^{-1}\tau_k' \textbf{d}, \textbf{V}_h)_{\Omega_k}=\sum_{k=1}^{n_{el}}(\tau_3^{-1}\tau_3'd_4, d_h)_{\Omega_k}
\end{split}
\end{equation}
Now taking all these terms together, (23) becomes
\begin{equation}
I_1+I_2+I_3+I_4+I_5+ I_6 =0, \hspace{1mm} \forall \hspace{1mm} \textbf{V}_h \in V_s^h \times V_s^h \times Q_s^h \times V_s^h \\
\end{equation}
This implies
\begin{multline}
I_1 + \int_{\Omega} \mu(c^n)\bigtriangledown E^{A,n,\theta}_{\textbf{u}} : \bigtriangledown \textbf{v}_h  +\int_{\Omega} \tilde{\bigtriangledown}E^{A,n,\theta}_{c} \cdot \bigtriangledown d_h +  \int_{\Omega} \sigma E^{A,n,\theta}_{u1}v_{1h}
+ \int_{\Omega} \sigma E^{A,n,\theta}_{u2}v_{2h} \\ +\int_{\Omega} \alpha E^{A,n,\theta}_{c} d_h 
= \int_{\Omega}(\bigtriangledown \cdot \textbf{v}_h)(E^{I,n,\theta}_{p} + E^{A,n,\theta}_{p})\\
-\int_{\Omega}(\bigtriangledown \cdot  E^{A,n,\theta}_{\textbf{u}}) q_h- \int_{\Omega} \tilde{\bigtriangledown}E^{I,n,\theta}_{c} \cdot \bigtriangledown d_h -\int_{\Omega}d_h \textbf{u}^n \cdot \bigtriangledown (E^{I,n,\theta}_{c}+E^{A,n,\theta}_{c}) \\
 -\int_{\Omega} \mu(c^n)\bigtriangledown E^{I,n,\theta}_{\textbf{u}} : \bigtriangledown \textbf{v}_h-I_3-I_4-I_5-I_6 \\ \hspace{1mm} \forall \hspace{1mm} \textbf{V}_h \in V_s^h \times V_s^h \times Q_s^h \times V_s^h
\end{multline} 
Now we will treat each term separately to find out the estimate. Before further proceeding let us mention an important consideration: since the above equation holds for all $\textbf{V}_h \in V_s^h \times V_s^h \times Q_s^h \times V_s^h$, therefore in each term we replace $v_{1h},v_{2h},q_h,d_h$ by $E^{A,n,\theta}_{u1},E^{A,n,\theta}_{u2},E^{A,n,\theta}_{p},E^{A,n,\theta}_{c}$ respectively as these auxiliary part of the errors belonging to their respective finite element spaces. From now onwards we will start derivation of each expression after considering the replacements directly.\vspace{2mm}\\
Let us start with $I_1$ as follows:
\begin{equation}
\begin{split}
(\frac{E^{A,n+1}_{c}-E^{A,n}_{c}}{dt},E^{A,n,\theta}_{c}) & = (\frac{E^{A,n+1}_{c}-E^{A,n}_{c}}{dt},\frac{1+\theta}{2} E^{A,n+1}_{c}+ \frac{1-\theta}{2} E^{A,n}_{c}) \\
&= \frac{1+\theta}{2 dt}(E^{A,n+1}_{c},E^{A,n+1}_{c})- \frac{1+\theta}{2 dt}(E^{A,n}_{c},E^{A,n+1}_{c})\\
& \quad + \frac{1-\theta}{2 dt}(E^{A,n+1}_{c},E^{A,n}_{c})-  \frac{1-\theta}{2 dt}(E^{A,n+1}_{c},E^{A,n+1}_{c})\\
& =\frac{1+\theta}{2 dt} \|E^{A,n+1}_{c}\|^2 - \frac{1-\theta}{2 dt} \|E^{A,n}_{c}\|^2 -\frac{\theta}{ dt} (E^{A,n}_{c},E^{A,n+1}_{c})\\
&= \frac{1}{2 dt}(\|E^{A,n+1}_{c}\|^2-\|E^{A,n}_{c}\|^2)+ \frac{\theta}{2 dt}(\|E^{A,n+1}_{c}\|^2-\|E^{A,n}_{c}\|^2)^2\\
& \geq \frac{1}{2 dt}(\|E^{A,n+1}_{c}\|^2-\|E^{A,n}_{c}\|^2)
\end{split}
\end{equation}
From $I_2$ we will select few terms to find out their lower bounds as follows:
\begin{equation}
\begin{split}
\int_{\Omega} \mu(c^n)\bigtriangledown E^{A,n,\theta}_{\textbf{u}} : \bigtriangledown  E^{A,n,\theta}_{\textbf{u}} & = \int_{\Omega}\mu(c^n)\{\sum_{i=1}^{2}(\frac{\partial E^{A,n,\theta}_{ui}}{\partial x})^2 + \sum_{i=1}^{2}(\frac{\partial E^{A,n,\theta}_{ui}}{\partial y})^2\}\\
& \geq \mu_l  \{\sum_{i=1}^{2}\int_{\Omega}(\frac{\partial E^{A,n,\theta}_{ui}}{\partial x})^2+ \sum_{i=1}^{2}\int_{\Omega} (\frac{\partial E^{A,n,\theta}_{ui}}{\partial y})^2 \}\\
& \geq \mu_l \{ \|\frac{\partial E^{A,n,\theta}_{u1}}{\partial x}\|^2 + \|\frac{\partial E^{A,n,\theta}_{u1}}{\partial y}\|^2 + \|\frac{\partial E^{A,n,\theta}_{u2}}{\partial x}\|^2 + \|\frac{\partial E^{A,n,\theta}_{u2}}{\partial y}\|^2 \}
\end{split}
\end{equation}
and
\begin{equation}
\begin{split}
 \int_{\Omega} \tilde{\bigtriangledown} E^{A,n,\theta}_{c} \cdot \bigtriangledown E^{A,n,\theta}_{c} & = \int_{\Omega}D_1 (\frac{\partial E^{A,n,\theta}_{c}}{\partial x})^2 + \int_{\Omega}D_2 (\frac{\partial E^{A,n,\theta}_{c}}{\partial y})^2\\
 & \geq D_l\{ \|\frac{\partial E^{A,n,\theta}_{c}}{\partial x}\|^2 + \| \frac{\partial E^{A,n,\theta}_{c}}{\partial y}\|^2 \}
\end{split}
\end{equation}
where $D_l$= min $\{ \underset{\Omega}{inf} D_1, \underset{\Omega}{inf} D_2  \}$. \\
Another few terms of $I_2$ can be easily simplified as,
\begin{equation}
\begin{split}
\int_{\Omega} \sigma E^{A,n,\theta}_{u1}E^{A,n,\theta}_{u1} &= \sigma \|E^{A,n,\theta}_{u1}\|^2\\ \int_{\Omega} \sigma E^{A,n,\theta}_{u2}E^{A,n,\theta}_{u2} &= \sigma \|E^{A,n,\theta}_{u2}\|^2 \\
\int_{\Omega} \alpha E^{A,n,\theta}_{c} E^{A,n,\theta}_{c} & = \alpha \|E^{A,n,\theta}_{c}\|^2 \\
\int_{\Omega}(\bigtriangledown \cdot E^{A,n,\theta}_\textbf{u})(E^{I,n,\theta}_{p} + E^{A,n,\theta}_{p})-\int_{\Omega}(\bigtriangledown \cdot  E^{A,n,\theta}_{\textbf{u}}) E^{A,n,\theta}_{p}  &=  \int_{\Omega}(\bigtriangledown \cdot E^{A,n,\theta}_\textbf{u})E^{I,n,\theta}_{p}
\end{split}
\end{equation}
Combining all these inequalities (32) becomes,
\begin{multline}
\frac{1}{2 dt}(\|E^{A,n+1}_{u1}\|^2-\|E^{A,n}_{u1}\|^2)+ \frac{1}{2 dt}(\|E^{A,n+1}_{u2}\|^2-\|E^{A,n}_{u2}\|^2)+ \frac{1}{2 dt}(\|E^{A,n+1}_{c}\|^2-\|E^{A,n}_{c}\|^2) \\
+ \mu_l \{ \|\frac{\partial E^{A,n,\theta}_{u1}}{\partial x}\|^2 + \|\frac{\partial E^{A,n,\theta}_{u1}}{\partial y}\|^2 + \|\frac{\partial E^{A,n,\theta}_{u2}}{\partial x}\|^2 + \|\frac{\partial E^{A,n,\theta}_{u2}}{\partial y}\|^2 \}+ \\
D_l\{ \|\frac{\partial E^{A,n,\theta}_{c}}{\partial x}\|^2 + \| \frac{\partial E^{A,n,\theta}_{c}}{\partial y}\|^2 \} + \sigma \|E^{A,n,\theta}_{u1}\|^2 +\sigma \|E^{A,n,\theta}_{u2}\|^2 +  \alpha \|E^{A,n,\theta}_{c}\|^2 \\
\leq \int_{\Omega}(\bigtriangledown \cdot E^{A,n,\theta}_\textbf{u})E^{I,n,\theta}_{p}- \int_{\Omega} \tilde{\bigtriangledown}E^{I,n,\theta}_{c} \cdot \bigtriangledown E^{A,n,\theta}_c -\int_{\Omega}E^{A,n,\theta}_c \textbf{u} \cdot \bigtriangledown E^{A,n,\theta}_{c}\\ -\int_{\Omega}E^{A,n,\theta}_c \textbf{u} \cdot \bigtriangledown E^{I,n,\theta}_{c} 
-\int_{\Omega} \mu(c^n)\bigtriangledown E^{I,n,\theta}_{\textbf{u}} : \bigtriangledown E^{A,n,\theta}_\textbf{u}-I_3-I_4-I_5-I_6
\end{multline}
Now we will find upper bounds of the terms in the RHS of the above equation.We will use Cauchy-Schwarz and Young's inequality to reach at the desired bounds.Let us start with the first term as follows:\vspace{1mm}\\
\begin{equation}
\begin{split}
\int_{\Omega}(\bigtriangledown \cdot E^{A,n,\theta}_\textbf{u})E^{I,n,\theta}_{p} &= \int_{\Omega} (\frac{\partial E^{A,n,\theta}_{u1}}{\partial x}+\frac{\partial E^{A,n,\theta}_{u2}}{\partial x})E^{I,n,\theta}_{p}\\
& \quad (applying \hspace{1mm} Cauchy-Schwarz \hspace{1mm} inequality) \\
& \leq (\|\frac{\partial E^{A,n,\theta}_{u1}}{\partial x}\| + \|\frac{\partial E^{A,n,\theta}_{u2}}{\partial y}\|) \|E^{I,n,\theta}_{p}\| \\
& \quad (applying \hspace{1mm} Young's \hspace{1mm} inequality \hspace{1mm} for \hspace{1mm} each \hspace{1mm} of \hspace{1mm} the \hspace{1mm} two \hspace{1mm} terms)\\
& \leq \frac{1}{2 \epsilon_1}(\|\frac{\partial E^{A,n,\theta}_{u1}}{\partial x}\|^2 + \|\frac{\partial E^{A,n,\theta}_{u2}}{\partial y}\|^2) + \epsilon_1 \|E^{I,n,\theta}_{p}\|^2 \\
& \leq \frac{1}{2 \epsilon_1}(\|\frac{\partial E^{A,n,\theta}_{u1}}{\partial x}\|^2 + \|\frac{\partial E^{A,n,\theta}_{u2}}{\partial y}\|^2) + \epsilon_1 (\frac{1+\theta}{2}\|E^{I,n+1}_{p}\| + \frac{1-\theta}{2}\|E^{I,n}_{p}\|)^2 \\
& \leq \frac{1}{2 \epsilon_1}(\|\frac{\partial E^{A,n,\theta}_{u1}}{\partial x}\|^2 + \|\frac{\partial E^{A,n,\theta}_{u2}}{\partial y}\|^2) + \epsilon_1 C^2 h^2(\frac{1+\theta}{2}\mid p^{n+1}\mid_1 + \frac{1-\theta}{2}\mid p^n \mid_1)^2 \\
\end{split}
\end{equation}
Similarly for each term we will use $Cauchy-Schwarz $ inequality and $ Young's $ inequality  wherever it will be needed, but without mentioning about them now onwards. \\
Second term
\begin{equation}
\begin{split}
- \int_{\Omega} \tilde{\bigtriangledown}E^{I,n,\theta}_{c} \cdot \bigtriangledown E^{A,n,\theta}_c & = -\int_{\Omega}(D_1 \frac{\partial E^{I,n,\theta}_{c}}{\partial x} \frac{\partial E^{A,n,\theta}_{c}}{\partial x}+ D_2 \frac{\partial E^{I,n,\theta}_{c}}{\partial y} \frac{\partial E^{A,n,\theta}_{c}}{\partial y} )\\
& \leq D_m (\|\frac{\partial E^{I,n,\theta}_{c}}{\partial x}\| \|\frac{\partial E^{A,n,\theta}_{c}}{\partial x}\|+ \|\frac{\partial E^{I,n,\theta}_{c}}{\partial y}\| \|\frac{\partial E^{A,n,\theta}_{c}}{\partial y}\|)\\
& \leq \frac{D_m}{2 \epsilon_2}(\|\frac{\partial E^{A,n,\theta}_{c}}{\partial x}\|^2 + \|\frac{\partial E^{A,n,\theta}_{c}}{\partial y}\|^2) + \frac{D_m \epsilon_2}{2} (\|\frac{\partial E^{I,n,\theta}_{c}}{\partial x}\|^2 + \|\frac{\partial E^{I,n,\theta}_{c}}{\partial y}\|^2) \\
&= \frac{D_m}{2 \epsilon_2}(\|\frac{\partial E^{A,n,\theta}_{c}}{\partial x}\|^2 + \|\frac{\partial E^{A,n,\theta}_{c}}{\partial y}\|^2)+  \frac{D_m \epsilon_2}{2} \mid E^{I,n,\theta}_c \mid_1^2\\
& \leq \frac{D_m}{2 \epsilon_2}(\|\frac{\partial E^{A,n,\theta}_{c}}{\partial x}\|^2 + \|\frac{\partial E^{A,n,\theta}_{c}}{\partial y}\|^2) + \frac{D_m \epsilon_2}{2} C^2 h^2 (\frac{1+\theta}{2}\mid c^{n+1} \mid_2+\frac{1-\theta}{2}\mid c^n \mid_2)^2
\end{split}
\end{equation}
where $D_m$= max $\{ \underset{\Omega}{sup} D_1, \underset{\Omega}{sup} D_2 \}$ \\
Next term,
\begin{equation}
\begin{split}
-\int_{\Omega}E^{A,n,\theta}_c \textbf{u}^n \cdot \bigtriangledown E^{A,n,\theta}_{c} & = -\int_\Omega(u_1^n E^{A,n,\theta}_c \frac{\partial E^{A,n,\theta}_c}{\partial x} + u_2^n E^{A,n,\theta}_c \frac{\partial E^{A,n,\theta}_c}{\partial y})\\
& \leq \underset{\Omega}{sup}\mid u_1^n \mid  \|E^{A,n,\theta}_c\| \|\frac{\partial E^{A,n,\theta}_c}{\partial x}\| + \underset{\Omega}{sup}\mid u_2^n \mid \|E^{A,n,\theta}_c\| \|\frac{\partial E^{A,n,\theta}_c}{\partial y}\| \\
& = (C_1^n \|\frac{\partial E^{A,n,\theta}_c}{\partial x}\| + C_2^n \|\frac{\partial E^{A,n,\theta}_c}{\partial y}\|) \| E^{A,n,\theta}_c\|\\
& \leq \frac{1}{2 \epsilon_3}(C_1^n\|\frac{\partial E^{A,n,\theta}_{c}}{\partial x}\|^2 + C_2^n\|\frac{\partial E^{A,n,\theta}_{c}}{\partial y}\|^2)+  \frac{\epsilon_3}{2}(C_1^n+C_2^n) \| E^{A,n,\theta}_c\|^2
\end{split}
\end{equation}
where $C_1^n$= $\underset{\Omega}{sup}$ $\mid u_1^n \mid$ and $C_2^n$= $\underset{\Omega}{sup}$ $\mid u_2^n \mid$. \\
Similarly the next term 
\begin{equation}
\begin{split}
-\int_{\Omega}E^{A,n,\theta}_c \textbf{u}^n \cdot \bigtriangledown E^{I,n,\theta}_{c} & \leq \frac{1}{2 \epsilon_3}(C_1^n\|\frac{\partial E^{I,n,\theta}_{c}}{\partial x}\|^2 + C_2^n\|\frac{\partial E^{I,n,\theta}_{c}}{\partial y}\|^2)+  \frac{\epsilon_3}{2}(C_1^n+C_2^n) \| E^{A,n,\theta}_c\|^2 \\
& \leq \frac{C_1^n+C_2^n}{2 \epsilon_3} \mid E^{I,n,\theta}_c \mid_1^2 + \frac{\epsilon_3}{2}(C_1^n+C_2^n) \| E^{A,n,\theta}_c\|^2 \\
& \leq \frac{C_1^n+C_2^n}{2 \epsilon_3} C^2 h^2 (\frac{1+\theta}{2}\mid c^{n+1} \mid_2 + \frac{1-\theta}{2} \mid c^n \mid_2)^2+ \frac{\epsilon_3}{2}(C_1^n+C_2^n) \| E^{A,n,\theta}_c\|^2 
\end{split}
\end{equation}
The next term,
\begin{equation}
\begin{split}
-\int_{\Omega} \mu(c^n)\bigtriangledown E^{I,n,\theta}_{\textbf{u}} : \bigtriangledown E^{A,n,\theta}_\textbf{u} & = - \int_{\Omega}\mu(c^n)(\frac{\partial E^{I,n,\theta}_{u1}}{\partial x} \frac{\partial E^{A,n,\theta}_{u1}}{\partial x} +\frac{\partial E^{I,n,\theta}_{u1}}{\partial y} \frac{\partial E^{A,n,\theta}_{u1}}{\partial y})\\
& \quad - \int_{\Omega}\mu(c^n)(\frac{\partial E^{I,n,\theta}_{u2}}{\partial x} \frac{\partial E^{A,n,\theta}_{u2}}{\partial x} +\frac{\partial E^{I,n,\theta}_{u2}}{\partial y} \frac{\partial E^{A,n,\theta}_{u2}}{\partial y})\\
& \leq \mu_u(\|\frac{\partial E^{I,n,\theta}_{u1}}{\partial x}\| \|\frac{\partial E^{A,n,\theta}_{u1}}{\partial x}\|+\|\frac{\partial E^{I,n,\theta}_{u1}}{\partial y}\| \|\frac{\partial E^{A,n,\theta}_{u1}}{\partial y}\| +\\
& \quad \|\frac{\partial E^{I,n,\theta}_{u2}}{\partial x}\| \|\frac{\partial E^{A,n,\theta}_{u2}}{\partial x}\|+ \|\frac{\partial E^{I,n,\theta}_{u2}}{\partial y}\| \|\frac{\partial E^{A,n,\theta}_{u2}}{\partial y}\|) \\
& \leq \frac{\epsilon_4 \mu_u}{2} \sum_{i=1}^{2} (\|\frac{\partial E^{I,n,\theta}_{ui}}{\partial x} \|^2+ \|\frac{\partial E^{I,n,\theta}_{ui}}{\partial y} \|^2)+\\
& \quad \frac{\mu_u}{2 \epsilon_4} \sum_{i=1}^{2} (\|\frac{\partial E^{A,n,\theta}_{ui}}{\partial x} \|^2+ \|\frac{\partial E^{A,n,\theta}_{ui}}{\partial y} \|^2)\\
& \leq \epsilon_4 \mu_u \sum_{i=1}^{2} C^2 h^2(\frac{1+\theta}{2}\mid u_i^{n+1}\mid_2+ \frac{1-\theta}{2} \mid u_i^n\mid_2)^2+\\
& \quad \frac{\mu_u}{2 \epsilon_4} \sum_{i=1}^{2} (\|\frac{\partial E^{A,n,\theta}_{ui}}{\partial x} \|^2+ \|\frac{\partial E^{A,n,\theta}_{ui}}{\partial y} \|^2)
\end{split}
\end{equation}
Now we will find bounds for each remaining term of $I_3$. Before going to further calculations let us mention an important observation:\vspace{2mm}\\
\begin{observation}
According to the choice of the finite element spaces $V_s^h$ and $Q_s^h$, we can clearly say that over each element sub-domain every function belonging to that spaces and their first and second order derivatives all are bounded functions. We can always find positive finite real numbers to bound each of the functions over element sub-domain. We will use this fact for several times further. 
\end{observation}
Let us take the first term of $(-I_3)$ along with earlier mentioned replacements. $I_3$ has four terms and we will find bounds for each of them separately. Let us denote the terms by $I_3^1, I_3^2, I_3^3, I_3^4$ respectively. Here we start with $I_3^1$,

\begin{multline}
-I_3^1= -\sum_{k=1}^{n_{el}} (\tau_1'( - \mu(c^n) \Delta (E^{I,n,\theta}_{u1}+E_{u1}^{A,n,\theta}) + \sigma (E^{I,n,\theta}_{u1}+E_{u1}^{A,n,\theta}) + \frac{\partial( E^{I,n,\theta}_{p})}{\partial x} + \\
\quad  \frac{\partial( E^{A,n,\theta}_{p})}{\partial x} ),\mu(c^n) \Delta E_{u1}^{A,n,\theta} - \sigma E_{u1}^{A,n,\theta} + \frac{\partial E_{p}^{A,n,\theta}}{\partial x})_{\Omega_k}\\
 =\tau_1' (\mu(c^n) \Delta (E^{I,n,\theta}_{u1}+E_{u1}^{A,n,\theta}) - \sigma (E^{I,n,\theta}_{u1}+E_{u1}^{A,n,\theta})- \frac{\partial( E^{I,n,\theta}_{p})}{\partial x} - 
 \frac{\partial( E^{A,n,\theta}_{p})}{\partial x} ,\\
  \quad \mu(c^n) \Delta E_{u1}^{A,n,\theta} - \sigma E_{u1}^{A,n,\theta} + \frac{\partial E_{p}^{A,n,\theta}}{\partial x})_{\tilde{\Omega}}\\
 =\tau_1' (\mu(c^n) \Delta E^{I,n,\theta}_{u1} - \sigma E^{I,n,\theta}_{u1}- \frac{\partial( E^{I,n,\theta}_{p})}{\partial x} ,\mu(c^n) \Delta E_{u1}^{A,n,\theta} - \sigma E_{u1}^{A,n,\theta} + \frac{\partial E_{p}^{A,n,\theta}}{\partial x})_{\tilde{\Omega}} \\
 \quad + \tau_1' (\mu(c^n) \Delta E^{A,n,\theta}_{u1} - \sigma E^{A,n,\theta}_{u1}- \frac{\partial( E^{A,n,\theta}_{p})}{\partial x},\mu(c^n) \Delta E_{u1}^{A,n,\theta} - \sigma E_{u1}^{A,n,\theta} + \frac{\partial E_{p}^{A,n,\theta}}{\partial x})_{\tilde{\Omega}} \\
\end{multline}
We calculate the bounds for the above two terms separately. The calculation for the first part is as follows:
\begin{multline}
\tau_1' (\mu(c^n) \Delta E^{I,n,\theta}_{u1} - \sigma E^{I,n,\theta}_{u1}- \frac{\partial( E^{I,n,\theta}_{p})}{\partial x} ,\mu(c^n) \Delta E_{u1}^{A,n,\theta} - \sigma E_{u1}^{A,n,\theta} + \frac{\partial E_{p}^{A,n,\theta}}{\partial x})_{\tilde{\Omega}} \\
=\sum_{k=1}^{n_{el}} \tau_1 ( \mu(c^n)^2 \Delta E^{I,n,\theta}_{u1} \Delta E_{u1}^{A,n,\theta}- \sigma \mu(c^n) E^{I,n,\theta}_{u1} \Delta E_{u1}^{A,n,\theta} - \mu(c^n)\frac{\partial E_{p}^{I,n,\theta}}{\partial x} \Delta E_{u1}^{A,n,\theta} \\
- \sigma \mu(c^n) E^{A,n,\theta}_{u1} \Delta E_{u1}^{I,n,\theta} + \sigma^2 E^{A,n,\theta}_{u1} E^{I,n,\theta}_{u1} + \sigma E^{A,n,\theta}_{u1} \frac{\partial E_{p}^{I,n,\theta}}{\partial x} + \\ 
\mu(c^n) \Delta E^{I,n,\theta}_{u1}
 \frac{\partial E_{p}^{A,n,\theta}}{\partial x} - \sigma E^{I,n,\theta}_{u1} \frac{\partial E_{p}^{A,n,\theta}}{\partial x} - \frac{\partial E_{p}^{A,n,\theta}}{\partial x} \frac{\partial E_{p}^{I,n,\theta}}{\partial x} )_{\Omega_k}\\
 \leq \sum_{k=1}^{n_{el}} \mid \tau_1 \mid (\mu_u^2 \|\Delta E^{I,n,\theta}_{u1}\|_k \|\Delta E^{A,n,\theta}_{u1}\|_k + \sigma \mu_u \|E^{I,n,\theta}_{u1}\|_k \|\Delta E^{A,n,\theta}_{u1}\|_k + 
 \mu_u\\
  \|\frac{\partial E_{p}^{I,n,\theta}}{\partial x}\|_k \| \Delta E_{u1}^{A,n,\theta}\|_k + \sigma \mu_u \| E^{A,n,\theta}_{u1}\|_k \| \Delta E_{u1}^{I,n,\theta}\|_k+ \sigma^2 \|  E^{A,n,\theta}_{u1}\|_k \\
  \| E_{u1}^{I,n,\theta}\|_k + \sigma \| E^{A,n,\theta}_{u1}\|_k \|\frac{\partial E_{p}^{I,n,\theta}}{\partial x} \|_k + \mu_u \|\Delta E^{I,n,\theta}_{u1}\|_k \|\frac{\partial E_{p}^{A,n,\theta}}{\partial x}\|_k + \\
\sigma \|E^{I,n,\theta}_{u1}\|_k \|\frac{\partial E_{p}^{A,n,\theta}}{\partial x}\|_k + \|\frac{\partial E_{p}^{A,n,\theta}}{\partial x}\|_k \|\frac{\partial E_{p}^{I,n,\theta}}{\partial x}\|_k) \\
 \end{multline}
Let $B_{1k},B_{2k},B_{3k}$ be the bounds on $E^{A,n,\theta}_{u1}, \Delta E^{A,n,\theta}_{u1}, \frac{\partial E^{A,n,\theta}_{p}}{\partial x}$ respectively on each element sub domain under the above observation 1 and $C_{\tau_1}$ be the maximum numerical value of $\tau_1$ over the domain $\Omega$ .

 \begin{multline}
\leq C_{\tau_1} \sum_{k=1}^{n_{el}} (\mu_u^2  B_{2k} \|\Delta E^{I,n,\theta}_{u1}\|_k  + \sigma \mu_u B_{2k} \|E^{I,n,\theta}_{u1}\|_k  + \mu_u B_{2k}
\mid E_{p}^{I,n,\theta} \mid_{1,k} + \\
 \sigma \mu_u B_{1k} \| \Delta E_{u1}^{I,n,\theta}\|_k+ \sigma^2 B_{1k} \| E_{u1}^{I,n,\theta}\|_k + \sigma B_{1k} \mid E_{p}^{I,n,\theta}\mid_{1,k} + \\
 \mu_u  B_{3k} \|\Delta E^{I,n,\theta}_{u1}\|_k + \sigma B_{3k} \|E^{I,n,\theta}_{u1}\|_k + B_{3k} \mid E_{p}^{I,n,\theta} \mid_{1,k})\\
\end{multline}
The expressions in the last line are obtained by applying bounds  on  the  members of  finite  element  spaces  over  each sub-domain. \\
Now  applying Inverse Inequality on domain $\Omega $ under the required assumption
\begin{multline}
\leq C_{\tau_1}  \{(\sum_{k=1}^{n_{el}}(\mu_u^2 B_{2k} + \sigma \mu_u B_{1k} +\mu_u B_{3k} ))h^{-1} \mid E^{I,n,\theta}_{u1}\mid_1 + (\sum_{k=1}^{n_{el}} \sigma (\mu_u B_{2k}+\sigma B_{1k}+\\
  B_{3k})) \|E^{I,n,\theta}_{u1} \| 
+ (\sum_{k=1}^{n_{el}}(\mu_u B_{2k}+ \sigma B_{1k}+ B_{3k}))\mid E_{p}^{I,n,\theta} \mid_1 \}\\
\leq C_{\tau_1}  \{ (\sum_{k=1}^{n_{el}}(\mu_u^2 B_{2k} + \sigma \mu_u B_{1k} +\mu_u B_{3k} ))C (\frac{1+\theta}{2} \mid u_1^{n+1} \mid_2 + \frac{1-\theta}{2}\mid u_1^n \mid_2) + \\
(\sum_{k=1}^{n_{el}}\sigma (\mu_u B_{2k}+\sigma B_{1k}+B_{3k})) C h^2 (\frac{1+\theta}{2} \mid u_1^{n+1} \mid_2 + \frac{1-\theta}{2}\mid u_1^n \mid_2)+ \\
(\sum_{k=1}^{n_{el}}(\mu_u B_{2k}+ \sigma B_{1k}+ B_{3k})) C h (\frac{1+\theta}{2} \mid p^{n+1} \mid_1 + \frac{1-\theta}{2}\mid p^n \mid_1) \}
\end{multline}
This completes the first part. Now we see that the second part has alike expression with auxiliary error terms in the place of interpolation error terms. Hence proceeding in the same way as above and applying bounds for elements belonging to $V_s^h$ and $Q_s^h$ spaces we will bound the second part as follows:
\begin{multline}
\tau_1' (\mu(c) \Delta E^{A,n,\theta}_{u1} - \sigma E^{A,n,\theta}_{u1}- \frac{\partial( E^{A,n,\theta}_{p})}{\partial x},\mu(c) \Delta E_{u1}^{A,n,\theta} - \sigma E_{u1}^{A,n,\theta} + \frac{\partial E_{p}^{A,n,\theta}}{\partial x})_{\tilde{\Omega}} \\
\leq C_{\tau_1} \sum_{k=1}^{n_{el}} (\mu_u^2 B_{2k}^2 + 2 \sigma \mu_u B_{1k} B_{2k} + \mu_u B_{2k} B_{3k} + \sigma^2 B_{1k}^2 + \sigma B_{1k} B_{3k} +\\
\mu_u B_{3k} B_{2k} + \sigma B_{1k} B_{3k} +B_{3k}^2 )\\
\leq C_{\tau_1} \sum_{k=1}^{n_{el}} M_{1k} \hspace{80mm}
\end{multline}
where $M_{1k}$ denotes the big sum of the constants. Combining all these results and putting into (43) we will have

\begin{equation}
\begin{split}
-I_3^1 & \leq  C_{\tau_1} \{ (\sum_{k=1}^{n_{el}}(\mu_u^2 B_{2k} + \sigma \mu_u B_{1k} +\mu_u B_{3k} ))C (\frac{1+\theta}{2} \mid u_1^{n+1} \mid_2 +  \\
& \quad \frac{1-\theta}{2}\mid u_1^n \mid_2) + (\sum_{k=1}^{n_{el}}\sigma (\mu_u B_{2k}+\sigma B_{1k}+B_{3k})) C h^2 (\frac{1+\theta}{2} \mid u_1^{n+1} \mid_2 +  \\
& \quad \frac{1-\theta}{2}\mid u_1^n \mid_2)+ (\sum_{k=1}^{n_{el}}(\mu_u B_{2k}+ \sigma B_{1k}+ B_{3k})) C h (\frac{1+\theta}{2} \mid p^{n+1} \mid_1 +  \\
& \quad \frac{1-\theta}{2}\mid p^n \mid_1) \}
+ C_{\tau_1} \sum_{k=1}^{n_{el}} M_{1k}
\end{split}
\end{equation} 
This completes the derivation of bound on the first term of $(-I_3)$. Now we see that the second term of $I_3$ in (27) is exactly similar to its first term, only the subscripts  are different that is $u_2$ replaces $u_1$ in subscript. Therefore considering the constants $B_{1k}',B_{2k}',B_{3k}'$ as the bounds for $E^{A,n,\theta}_{u2}, \Delta E^{A,n,\theta}_{u2},\frac{\partial E_{p}^{A,n,\theta}}{\partial y} $ respectively on each element sub domain, we can bound the term as follows:
\begin{equation}
\begin{split}
-I_3^2 & \leq  C_{\tau_1} \mid \{ (\sum_{k=1}^{n_{el}}(\mu_u^2 B_{2k}' + \sigma \mu_u B_{1k}' +\mu_u B_{3k}' ))C (\frac{1+\theta}{2} \mid u_2^{n+1} \mid_2 +  \\
& \quad \frac{1-\theta}{2}\mid u_2^n \mid_2) + (\sum_{k=1}^{n_{el}}\sigma (\mu_u B_{2k}'+\sigma B_{1k}'+B_{3k}')) C h^2 (\frac{1+\theta}{2} \mid u_2^{n+1} \mid_2 +  \\
& \quad \frac{1-\theta}{2}\mid u_2^n \mid_2)+ (\sum_{k=1}^{n_{el}}(\mu_u B_{2k}'+ \sigma^2 B_{1k}'+ B_{3k}')) C h (\frac{1+\theta}{2} \mid p^{n+1} \mid_1 +  \\
& \quad \frac{1-\theta}{2}\mid p^n \mid_1) \}
+ C_{\tau_1} \sum_{k=1}^{n_{el}} M_{2k}
\end{split}
\end{equation}
Like $M_{1k}$, $M_{2k}$ denotes the big sum associated with second term of $I_3^2$.
Now we are going to derive bounds for the third term of $I_3$ as follows:

\begin{equation}
\begin{split}
-I_3^3 & = -\tau_2' \sum_{k=1}^{n_{el}}(\bigtriangledown \cdot (E^{I,n,\theta}_{\textbf{u}}+E_{\textbf{u}}^{A,n,\theta}), \bigtriangledown \cdot E_{\textbf{u}}^{A,n,\theta} )_{\Omega_k} \\
& = -\tau_2 \sum_{k=1}^{n_{el}}(\bigtriangledown \cdot E^{I,n,\theta}_{\textbf{u}}, \bigtriangledown \cdot E_{\textbf{u}}^{A,n,\theta} )_{\Omega_k} - \tau_2' \sum_{k=1}^{n_{el}} (\bigtriangledown \cdot E^{A,n,\theta}_{\textbf{u}}, \bigtriangledown \cdot E_{\textbf{u}}^{A,n,\theta} )_{\Omega_k} \\
\end{split}
\end{equation}
\begin{equation}
\begin{split}
& = -\tau_2 \sum_{k=1}^{n_{el}} \int_{\Omega_k} (\frac{\partial E^{I,n,\theta}_{u1}}{\partial x} \frac{\partial E^{A,n,\theta}_{u1}}{\partial x}+ \frac{\partial E^{I,n,\theta}_{u2}}{\partial y} \frac{\partial E^{A,n,\theta}_{u2}}{\partial y} + \frac{\partial E^{I,n,\theta}_{u1}}{\partial x} \frac{\partial E^{A,n,\theta}_{u2}}{\partial y}+ \\
& \quad \frac{\partial E^{A,n,\theta}_{u1}}{\partial x} \frac{\partial E^{I,n,\theta}_{u2}}{\partial y}+
2 \frac{\partial E^{A,n,\theta}_{u1}}{\partial x} \frac{\partial E^{A,n,\theta}_{u2}}{\partial y}+ (\frac{\partial E^{A,n,\theta}_{u1}}{\partial x})^2 + (\frac{\partial E^{A,n,\theta}_{u2}}{\partial y})^2)    \\
& \leq \mid \tau_2 \mid \sum_{k=1}^{n_{el}} (\|\frac{\partial E^{I,n,\theta}_{u1}}{\partial x}\|_k \|\frac{\partial E^{A,n,\theta}_{u1}}{\partial x}\|_k+ \|\frac{\partial E^{I,n,\theta}_{u2}}{\partial y}\|_k \|\frac{\partial E^{A,n,\theta}_{u2}}{\partial y}\|_k + \|\frac{\partial E^{A,n,\theta}_{u1}}{\partial x}\|_k^2 + \\
& \quad \|\frac{\partial E^{I,n,\theta}_{u1}}{\partial x}\|_k \|\frac{\partial E^{A,n,\theta}_{u2}}{\partial y} \|_k+ \|\frac{\partial E^{A,n,\theta}_{u1}}{\partial x} \|_k \|\frac{\partial E^{I,n,\theta}_{u2}}{\partial y}\|_k + 2\|\frac{\partial E^{A,n,\theta}_{u1}}{\partial x} \|_k \|\frac{\partial E^{A,n,\theta}_{u2}}{\partial y} \|_k + \\
& \quad \|\frac{\partial E^{A,n,\theta}_{u2}}{\partial y}\|_k^2)\\
& \leq \mid \tau_2 \mid \sum_{k=1}^{n_{el}} ((B_{4k}+ B_{5k}')(\|\frac{\partial E^{I,n,\theta}_{u1}}{\partial x}\|_k+ \|\frac{\partial E^{I,n,\theta}_{u2}}{\partial y}\|_k)+ B_{4k}^2+2B_{4k} B_{5k}' +B_{5k}'^2)\\
& \leq \mid \tau_2 \mid \sum_{k=1}^{n_{el}}( B_{4k} +B_{5k}')^2 + \mid \tau_2 \mid ( \sum_{k=1}^{n_{el}} (B_{4k}+B_{5k}')) (\mid E^{I,n,\theta}_{u1} \mid_1 + \mid E^{I,n,\theta}_{u2} \mid_1) \\
& \leq C_{\tau_2} \sum_{k=1}^{n_{el}}( B_{4k} +B_{5k}')^2+ C_{\tau_2} ( \sum_{k=1}^{n_{el}}( B_{4k}+B_{5k}'))C h \{(\frac{1+\theta}{2} \mid u_1^{n+1} \mid_2 + \frac{1-\theta}{2} \mid u_1^{n} \mid_2 )+ \\
& \quad  (\frac{1+\theta}{2} \mid u_2^{n+1} \mid_2 + \frac{1-\theta}{2} \mid u_2^{n} \mid_2 )\}
\end{split}
\end{equation}
where the constants $B_{4k},B_{5k},B_{4k}'$ and $B_{5k}'$ are bounds on $ \frac{\partial E^{A,n,\theta}_{u1}}{\partial x} , \frac{\partial E^{A,n,\theta}_{u1}}{\partial y}, \frac{\partial E^{A,n,\theta}_{u2}}{\partial x} $ and $\frac{\partial E^{A,n,\theta}_{u2}}{\partial y}$ respectively on each element sub domain and $C_{\tau_2}$ is the maximum numerical value for $\tau_2$ over $\Omega$. Now we will focus on the fourth term of $I_3$. We will divide $I_3^4$ into three parts $P_1,P_2$ and $P_3$ and then calculate bounds for each of them separately. 

\begin{equation}
\begin{split}
-I_3^4 &= - \sum_{k=1}^{n_{el}} \tau_3'( \partial_t(E^{I,n}_{c}+E_{c}^{A,n}) - \bigtriangledown \cdot \tilde{\bigtriangledown}(E^{I,n,\theta}_{c}+E_{c}^{A,n,\theta}) + \textbf{u} \cdot \bigtriangledown (E^{I,n,\theta}_{c}+E_{c}^{A,n,\theta})\\
& \quad + \alpha (E^{I,n,\theta}_{c}+E_{c}^{A,n,\theta}),\bigtriangledown \cdot \tilde{\bigtriangledown} E_{c}^{A,n,\theta} +\textbf{u} \cdot \bigtriangledown E_{c}^{A,n,\theta} - \alpha E_{c}^{A,n,\theta} )_{\Omega_k} \\
& = \sum_{k=1}^{n_{el}} \tau_3'( \partial_t(E^{I,n}_{c}+E_{c}^{A,n}),-\bigtriangledown \cdot \tilde{\bigtriangledown} E_{c}^{A,n,\theta} -\textbf{u} \cdot \bigtriangledown E_{c}^{A,n,\theta} + \alpha E_{c}^{A,n,\theta} )_{\Omega_k} + \sum_{k=1}^{n_{el}} \tau_3' \\
& \quad  (\bigtriangledown \cdot \tilde{\bigtriangledown}E^{I,n,\theta}_{c} - \textbf{u} \cdot \bigtriangledown E^{I,n,\theta}_{c}- \alpha E^{I,n,\theta}_{c},\bigtriangledown \cdot \tilde{\bigtriangledown} E_{c}^{A,n,\theta} +\textbf{u} \cdot \bigtriangledown E_{c}^{A,n,\theta} -  \alpha E_{c}^{A,n,\theta} )_{\Omega_k} + \\
& \quad   \sum_{k=1}^{n_{el}} \tau_3' (\bigtriangledown \cdot \tilde{\bigtriangledown}E^{A,n,\theta}_{c} - \textbf{u} \cdot \bigtriangledown E^{A,n,\theta}_{c}- \alpha E^{A,n,\theta}_{c},\bigtriangledown \cdot \tilde{\bigtriangledown} E_{c}^{A,n,\theta}+ \textbf{u} \cdot \bigtriangledown E_{c}^{A,n,\theta} -  \\
& \quad \alpha E_{c}^{A,n,\theta} )_{\Omega_k}\\
& = P_1+ P_2 + P_3
\end{split}
\end{equation}
Let us start with $P_1$
\begin{equation}
\begin{split}
P_1 & = \sum_{k=1}^{n_{el}} \tau_3'( \partial_t(E^{I,n}_{c}+E_{c}^{A,n}),-\bigtriangledown \cdot \tilde{\bigtriangledown} E_{c}^{A,n,\theta} -\textbf{u} \cdot \bigtriangledown E_{c}^{A,n,\theta} + \alpha E_{c}^{A,n,\theta} )_{\Omega_k} \\
& = \sum_{k=1}^{n_{el}} \tau_3'\alpha( \partial_t E^{I,n}_{c}+ \partial_t E_{c}^{A,n}, E_{c}^{A,n,\theta})_{\Omega_k}- \sum_{k=1}^{n_{el}} \tau_3'(\partial_t E^{I,n}_{c}+ \partial_t E_{c}^{A,n},\\
& \quad \bigtriangledown \cdot \tilde{\bigtriangledown} E_{c}^{A,n,\theta} +\textbf{u} \cdot \bigtriangledown E_{c}^{A,n,\theta})_{\Omega_k}\\
&= \sum_{k=1}^{n_{el}} \tau_3'\alpha(\partial_t E_{c}^{A,n}, E_{c}^{A,n,\theta})_{\Omega_k} -
\sum_{k=1}^{n_{el}} \tau_3'(\partial_t E^{I,n,\theta}_{c}+ \partial_t E_{c}^{A,n,\theta}, D_1 \frac{\partial^2 E_{c}^{A,n,\theta}}{\partial x^2}+\\
& \quad D_2 \frac{\partial^2 E_{c}^{A,n,\theta}}{\partial y^2} + (u_1+ \frac{\partial D_1}{\partial x})\frac{\partial E_{c}^{A,n,\theta} }{\partial x} + (u_2+ \frac{\partial D_2}{\partial y})\frac{\partial E_{c}^{A,n,\theta} }{\partial y})_{\Omega_k}\\
& = \alpha\tau_3' \sum_{k=1}^{n_{el}} \int_{\Omega_k} \frac{E^{A,n+1}_{c}-E^{A,n}_{c}}{dt}E_{c}^{A,n,\theta}- \tau_3' \sum_{k=1}^{n_{el}} \int_{\Omega_k} \frac{E^{I,n+1}_{c}-E^{I,n}_{c}}{dt} (D_1 \frac{\partial^2 E_{c}^{A,n,\theta}}{\partial x^2}+\\
& \quad D_2 \frac{\partial^2 E_{c}^{A,n,\theta}}{\partial y^2} + (u_1+ \frac{\partial D_1}{\partial x})\frac{\partial E_{c}^{A,n,\theta} }{\partial x} + (u_2+ \frac{\partial D_2}{\partial y})\frac{\partial E_{c}^{A,n,\theta} }{\partial y})-\\
& \quad \tau_3' \sum_{k=1}^{n_{el}} \int_{\Omega_k} \frac{E^{A,n+1}_{u1}-E^{A,n}_{u1}}{dt} (D_1 \frac{\partial^2 E_{c}^{A,n,\theta}}{\partial x^2}+ D_2 \frac{\partial^2 E_{c}^{A,n,\theta}}{\partial y^2} + (u_1+ \frac{\partial D_1}{\partial x})\frac{\partial E_{c}^{A,n,\theta} }{\partial x} + \\
& \quad (u_2+ \frac{\partial D_2}{\partial y})\frac{\partial E_{c}^{A,n,\theta} }{\partial y}) \\
& \leq \frac{\alpha \mid  \tau_3'\mid }{dt} \sum_{k=1}^{n_{el}} \int_{\Omega_k}(E^{A,n+1}_{c}-E^{A,n}_{c}) E_{c}^{A,n,\theta}  + \frac{\mid  \tau_3'\mid}{dt} \sum_{k=1}^{n_{el}} \int_{\Omega_k}(E^{I,n+1}_{c}-E^{I,n}_{c}) \\
& \quad (D_{1m} \frac{\partial^2 E_{c}^{A,n,\theta}}{\partial x^2}+D_{2m}\frac{\partial^2 E_{c}^{A,n,\theta}}{\partial y^2} + D_{u1} \frac{\partial E_{c}^{A,n,\theta} }{\partial x}+ D_{u2} \frac{\partial E_{c}^{A,n,\theta} }{\partial y})+ \frac{\mid  \tau_3'\mid}{dt} \sum_{k=1}^{n_{el}}\\
& \quad  \int_{\Omega_k}(E^{A,n+1}_{c}-E^{A,n}_{c})(D_{1m}  \frac{\partial^2 E_{c}^{A,n,\theta}}{\partial x^2} + D_{2m}\frac{\partial^2 E_{c}^{A,n,\theta}}{\partial y^2} + D_{u1} \frac{\partial E_{c}^{A,n,\theta} }{\partial x}+  D_{u2} \frac{\partial E_{c}^{A,n,\theta} }{\partial y})\\
& \leq \frac{\alpha \mid  \tau_3'\mid }{dt} (\sum_{k=1}^{n_{el}} B_{6k}) (\|E^{A,n+1}_{c}\|^2-\|E^{A,n}_{c}\|^2)+\frac{\mid  \tau_3'\mid}{dt} \{\sum_{k=1}^{n_{el}} (D_{1m} B_{7k}+D_{2m}B_{7k}'+D_{u1} \\
& \quad B_{8k}+ D_{u2} B_{8k}')\}(\|E^{I,n+1}_{c}\|+\|E^{I,n}_{c}\|) + \frac{\mid  \tau_3'\mid}{dt} \{\sum_{k=1}^{n_{el}} (D_{1m} B_{7k}+D_{2m}B_{7k}'+D_{u1}B_{8k}+\\
& \quad D_{u2} B_{8k}')\}(\|E^{A,n+1}_{c}\|^2-\|E^{A,n}_{c}\|^2)\\
& \leq \frac{C_{\tau_3} T}{dt(T_0-C_{\tau_3})}\{ \sum_{k=1}^{n_{el}}(\alpha B_{6k} +D_{B_{1k}})\}(\|E^{A,n+1}_{c}\|^2-\|E^{A,n}_{c}\|^2)+\frac{C_{\tau_3}Ch^2}{(T_0-C_{\tau_3})} \{\sum_{k=1}^{n_{el}} D_{B_{1k}} \} \\
& \quad (\mid c^{n+1} \mid_2 + \mid c^n \mid_2)
\end{split}
\end{equation}
where the constants $B_{6k},B_{7k},B_{7k}',B_{8k}$ and $B_{8k}'$ are upper bounds on $E^{A,n,\theta}_{c},\frac{\partial^2 E^{A,n,\theta}_{c}}{\partial x^2}, \frac{\partial^2 E^{A,n,\theta}_{c}}{\partial y^2} ,\\ \frac{\partial E_{c}^{A,n,\theta}}{\partial x}$ and $\frac{\partial E_{c}^{A,n,\theta}}{\partial y}$ respectively on each element sub domain and $D_{1m},D_{2m}, {D}_{u1}, {D}_{u2}$ are maximum of the functions $D_1, D_2, (\frac{\partial D_1}{\partial x}+ u_1),(\frac{\partial D_2}{\partial y}+ u_2) $ respectively over $\Omega$. $C_{\tau_3}$ and $T_0$ are maximum bound of $\tau_3$ and minimum bound of time step $dt$ respectively. At the last line new notation $D_{B_{1k}}$ represents the big sum.
\begin{equation}
\begin{split}
P_2 & = \sum_{k=1}^{n_{el}} \tau_3' (\bigtriangledown \cdot \tilde{\bigtriangledown}E^{I,n,\theta}_{c} - \textbf{u} \cdot \bigtriangledown E^{I,n,\theta}_{c}- \alpha E^{I,n,\theta}_{c},\bigtriangledown \cdot \tilde{\bigtriangledown} E_{c}^{A,n,\theta} +\textbf{u} \cdot \bigtriangledown E_{c}^{A,n,\theta} - \alpha E_{c}^{A,n,\theta} )_{\Omega_k} \\
& = \sum_{k=1}^{n_{el}} \tau_3' (D_1 \frac{\partial^2 E^{I,n,\theta}_{c}}{\partial x^2} + D_2 \frac{\partial^2 E^{I,n,\theta}_{c}}{\partial y^2} + (\frac{\partial D_1}{\partial x}-u_1)\frac{\partial E^{I,n,\theta}_{c}}{\partial x}+ (\frac{\partial D_2}{\partial y}-u_2)\\
& \quad \frac{\partial E^{I,n,\theta}_{c}}{\partial y}-\alpha E^{I,n,\theta}_{c}, D_1 \frac{\partial^2 E^{A,n,\theta}_{c}}{\partial x^2} + D_2 \frac{\partial^2 E^{A,n,\theta}_{c}}{\partial y^2} + (\frac{\partial D_1}{\partial x}+u_1)\frac{\partial E^{A,n,\theta}_{c}}{\partial x}+ \\
& \quad (\frac{\partial D_2}{\partial y}+u_2)\frac{\partial E^{A,n,\theta}_{c}}{\partial y}-\alpha E^{A,n,\theta}_{c} )_{\Omega_k}\\
& = \sum_{k=1}^{n_{el}} \tau_3'(D_1^2 \frac{\partial^2 E^{I,n,\theta}_{c}}{\partial x^2} \frac{\partial^2 E^{A,n,\theta}_{c}}{\partial x^2}+ D_1 D_2 \frac{\partial^2 E^{I,n,\theta}_{c}}{\partial y^2}\frac{\partial^2 E^{A,n,\theta}_{c}}{\partial x^2}+D_1(\frac{\partial D_1}{\partial x}-u_1)\frac{\partial E^{I,n,\theta}_{c}}{\partial x}\\
& \quad \frac{\partial^2 E^{A,n,\theta}_{c}}{\partial x^2} + D_1(\frac{\partial D_2}{\partial x}-u_2)\frac{\partial E^{I,n,\theta}_{c}}{\partial y} \frac{\partial^2 E^{A,n,\theta}_{c}}{\partial x^2}-\alpha D_1 E^{I,n,\theta}_{c}\frac{\partial^2 E^{A,n,\theta}_{c}}{\partial x^2}+ D_1 D_2 \frac{\partial^2 E^{I,n,\theta}_{c}}{\partial x^2} \\
& \quad \frac{\partial^2 E^{A,n,\theta}_{c}}{\partial y^2}+D_2^2 \frac{\partial^2 E^{I,n,\theta}_{c}}{\partial y^2} \frac{\partial^2 E^{A,n,\theta}_{c}}{\partial y^2}+D_2(\frac{\partial D_1}{\partial x}-u_1)\frac{\partial E^{I,n,\theta}_{c}}{\partial x}\frac{\partial^2 E^{A,n,\theta}_{c}}{\partial y^2}+(\frac{\partial D_2}{\partial x}-u_2)\\
& \quad  D_2 \frac{\partial E^{I,n,\theta}_{c}}{\partial y} \frac{\partial^2 E^{A,n,\theta}_{c}}{\partial y^2}-\alpha D_2 E^{I,n,\theta}_{c} \frac{\partial^2 E^{A,n,\theta}_{c}}{\partial y^2} + D_1(\frac{\partial D_1}{\partial x}+u_1)\frac{\partial E^{A,n,\theta}_{c}}{\partial x} \frac{\partial^2 E^{I,n,\theta}_{c}}{\partial x^2}+ D_2\\
& \quad (\frac{\partial D_1}{\partial x}+u_1)\frac{\partial E^{A,n,\theta}_{c}}{\partial x} \frac{\partial^2 E^{I,n,\theta}_{c}}{\partial y^2}+(\frac{\partial D_1}{\partial x}^2-u_1^2)\frac{\partial E^{I,n,\theta}_{c}}{\partial x} \frac{\partial E^{A,n,\theta}_{c}}{\partial x}+ (\frac{\partial D_1}{\partial x}+u_1) (\frac{\partial D_2}{\partial x}-u_2)\\
& \quad \frac{\partial E^{A,n,\theta}_{c}}{\partial x} \frac{\partial E^{I,n,\theta}_{c}}{\partial y} -\alpha (\frac{\partial D_1}{\partial x}+u_1) E^{I,n,\theta}_{c} \frac{\partial E^{A,n,\theta}_{c}}{\partial x} + D_1 (\frac{\partial D_2}{\partial x}+u_2)\frac{\partial E^{A,n,\theta}_{c}}{\partial y} \frac{\partial^2 E^{I,n,\theta}_{c}}{\partial x^2} +  D_2 \\
& \quad  (\frac{\partial D_2}{\partial y}+u_2)\frac{\partial E^{A,n,\theta}_{c}}{\partial y} \frac{\partial^2 E^{I,n,\theta}_{c}}{\partial y^2} + (\frac{\partial D_1}{\partial x}-u_1) (\frac{\partial D_2}{\partial y}+u_2)\frac{\partial E^{I,n,\theta}_{c}}{\partial x} \frac{\partial E^{A,n,\theta}_{c}}{\partial y}+ (\frac{\partial D_2}{\partial y}^2-u_2^2)\\
& \quad \frac{\partial E^{I,n,\theta}_{c}}{\partial y} \frac{\partial E^{A,n,\theta}_{c}}{\partial y}- \alpha D_1 \frac{\partial^2 E^{I,n,\theta}_{c}}{\partial x^2} E^{A,n,\theta}_{c}- \alpha D_2 \frac{\partial^2 E^{I,n,\theta}_{c}}{\partial y^2} E^{A,n,\theta}_{c}- \alpha (\frac{\partial D_1}{\partial x}-u_1)\frac{\partial E^{I,n,\theta}_{c}}{\partial x} \\
& \quad E^{A,n,\theta}_{c} - \alpha (\frac{\partial D_2}{\partial x}-u_2)\frac{\partial E^{I,n,\theta}_{c}}{\partial y} E^{A,n,\theta}_{c} + \alpha^2 E^{I,n,\theta}_{c} E^{A,n,\theta}_{c} - \alpha (\frac{\partial D_2}{\partial y}+u_2)E^{I,n,\theta}_{c} \frac{\partial E^{A,n,\theta}_{c}}{\partial y})_{\Omega_k} \\
& \leq \sum_{k=1}^{n_{el}} \mid \tau_3' \mid(D_{1m}^2 B_{7k} \|\frac{\partial^2 E^{I,n,\theta}_{c}}{\partial x^2}\|_k + D_{1m}D_{2m} B_{7k}' \|\frac{\partial^2 E^{I,n,\theta}_{c}}{\partial y^2}\|_k + D_{1m} \bar{D}_{u1} B_{7k} \|\frac{\partial E^{I,n,\theta}_{c}}{\partial x}\|_k \\
& \quad D_{1m}\bar{D}_{u2} B_{7k} \|\frac{\partial E^{I,n,\theta}_{c}}{\partial x}\|_k + \alpha D_{1m} B_{7k} \|E^{I,n,\theta}_{c}\|_k + D_{1m}D_{2m} B_{7k}' \|\frac{\partial^2 E^{I,n,\theta}_{c}}{\partial x^2}\|_k + D_{2m}^2 B_{7k}' \\
\end{split}
\end{equation}
\begin{equation}
\begin{split}
& \quad \|\frac{\partial^2 E^{I,n,\theta}_{c}}{\partial y^2}\|_k + D_{2m} \bar{D}_{u1} B_{7k}' \|\frac{\partial^2 E^{I,n,\theta}_{c}}{\partial y^2}\|_k + D_{2m} \bar{D}_{u2} B_{7k}' \|\frac{\partial E^{I,n,\theta}_{c}}{\partial y}\|_k + \alpha D_{2m} B_{7k}' \|E^{I,n,\theta}_{c}\|_k +\\
& \quad D_{1m} D_{u1} B_{8k} \|\frac{\partial^2 E^{I,n,\theta}_{c}}{\partial x^2}\|_k + D_{2m} D_{u1} B_{8k} \|\frac{\partial^2 E^{I,n,\theta}_{c}}{\partial y^2}\|_k + \bar{D}_{u1} \bar{D}_{u2} B_{8k} \|\frac{\partial E^{I,n,\theta}_{c}}{\partial x}\|_k + D_{u1} \bar{D}_{u2} \\
& \quad B_{8k} \|\frac{\partial E^{I,n,\theta}_{c}}{\partial y}\|_k + \alpha D_{u1}  B_{8k} \|E^{I,n,\theta}_{c}\|_k + D_{1m} D_{u2} B_{8k}'  \|\frac{\partial^2 E^{I,n,\theta}_{c}}{\partial x^2}\|_k + D_{2m} D_{u2} B_{8k}'  \|\frac{\partial^2 E^{I,n,\theta}_{c}}{\partial y^2}\|_k \\
& \quad + \bar{D}_{u1} D_{u2} B_{8k}' \|\frac{\partial E^{I,n,\theta}_{c}}{\partial x}\|_k +  D_{u2} \bar{D}_{u2} B_{8k}' \|\frac{\partial E^{I,n,\theta}_{c}}{\partial y}\|_k + \alpha D_{1m} B_{6k} \|\frac{\partial^2 E^{I,n,\theta}_{c}}{\partial x^2}\|_k + \alpha  D_{2m} B_{6k} \\
& \quad \|\frac{\partial^2 E^{I,n,\theta}_{c}}{\partial y^2}\|_k + \alpha \bar{D}_{u1} B_{6k} \|\frac{\partial E^{I,n,\theta}_{c}}{\partial x}\|_k +
\alpha \bar{D}_{u2} B_{6k} \|\frac{\partial E^{I,n,\theta}_{c}}{\partial y}\|_k + \alpha^2 B_{6k} \| E^{I,n,\theta}_{c}\|_k + \alpha D_{u2} B_{8k}'\\
& \quad \|E^{I,n,\theta}_c\|_k) \\
& \leq \mid \tau_3' \mid \{ \sum_{k=1}^{n_{el}} (D_{1m}^2 B_{7k} +2D_{1m}D_{2m} B_{7k}' +D_{2m}^2 B_{7k}' +D_{2m} \bar{D}_{u1} B_{7k}'+ D_{1m} D_{u1} B_{8k}+\\
& \quad D_{2m} D_{u1} B_{8k} +D_{1m} D_{u2} B_{8k}' +D_{2m} D_{u2} B_{8k}' + \alpha D_{1m} B_{6k} + \alpha D_{2m} B_{6k})\} \mid E^{I,n,\theta}_c \mid_2 +  \mid \tau_3' \mid \\
& \quad \{ \sum_{k=1}^{n_{el}} (D_{1m} \bar{D}_{u1} B_{7k}+ D_{1m} \bar{D}_{u2} B_{7k}+D_{2m} \bar{D}_{u2} B_{7k}'+ \bar{D}_{u1} \bar{D}_{u2} B_{8k} + D_{ul} \bar{D}_{u2} B_{8k} + D_{u2} \bar{D}_{u1} \\
& \quad B_{8k}'+D_{u2} \bar{D}_{u2} B_{8k}' + \alpha  \bar{D}_{u1} B_{6k}+ \alpha  \bar{D}_{u2} B_{6k})\} \mid E^{I,n,\theta}_c \mid_1 + \mid \tau_3' \mid \{ \sum_{k=1}^{n_{el}} (\alpha D_{1m} B_{7k} + \alpha \\
& \quad D_{2m} B_{7k}' + \alpha D_{u1} B_{8k} + \alpha D_{u2} B_{8k}' + \alpha^2 B_{6k})\} \|E^{I,n,\theta}_c\| \\
& \leq \mid \tau_3' \mid \{ \sum_{k=1}^{n_{el}} (D_{1m}^2 B_{7k} +2D_{1m}D_{2m} B_{7k}' +D_{2l}^2 B_{7k}' +D_{2m} \bar{D}_{u1} B_{7k}'+ D_{1m} D_{u1} B_{8k}+\\
& \quad D_{2m} D_{u1} B_{8k} +D_{1m} D_{u2} B_{8k}' +D_{2m} D_{u2} B_{8k}' + \alpha D_{1m} B_{6k} + \alpha D_{2m} B_{6k})\} C(\frac{1+\theta}{2}\mid c^{n+1}\mid_2 + \\
& \quad \frac{1-\theta}{2} \mid c^n \mid_2) + \mid \tau_3' \mid \{ \sum_{k=1}^{n_{el}} (D_{1m} \bar{D}_{u1} B_{7k}+ D_{1m} \bar{D}_{u2} B_{7k}+D_{2m} \bar{D}_{u2} B_{7k}'+ \bar{D}_{u1} \bar{D}_{u2} B_{8k} +  \\
& \quad D_{u1}\bar{D}_{u2} B_{8k} + D_{u2} \bar{D}_{u1} B_{8k}'+D_{u2} \bar{D}_{u2} B_{8k}' + \alpha  \bar{D}_{u1} B_{6k}+ \alpha  \bar{D}_{u2} B_{6k})\} C h (\frac{1+\theta}{2}\mid c^{n+1}\mid_1+ \\
& \quad \frac{1-\theta}{2} \mid c^n \mid_1)+  \mid \tau_3' \mid \{ \sum_{k=1}^{n_{el}} (\alpha D_{1m} B_{7k} + \alpha D_{2m} B_{7k}' + \alpha D_{u1} B_{8k} + \alpha D_{u2} B_{8k}' + \alpha^2 B_{6k})\} C h^2 \\
& \quad (\frac{1+\theta}{2}\| c^{n+1}\|+\frac{1-\theta}{2} \| c^n \|) \\
& \leq \frac{C_{\tau_3}T}{(T_0-C_{\tau_3})} C \{ (\sum_{k=1}^{n_{el}} D_{B_{2k}})(\frac{1+\theta}{2}\mid c^{n+1}\mid_2 +\frac{1-\theta}{2} \mid c^n \mid_2)+h (\sum_{k=1}^{n_{el}} D_{B_{3k}})(\frac{1+\theta}{2}\mid c^{n+1}\mid_1  \\
& \quad +\frac{1-\theta}{2} \mid c^n \mid_1) + h^2 (\sum_{k=1}^{n_{el}} D_{B_{4k}})(\frac{1+\theta}{2}\| c^{n+1}\| +\frac{1-\theta}{2} \| c^n \|) \}
\end{split}
\end{equation}
where $D_{B_{2k}}, D_{B_{3k}}$ and $D_{B_{4k}}$ are denoting respectively the summations in which the notations $ \bar{D}_{u1}, \bar{D}_{u2}$ are the maximum of the functions $ (\frac{\partial D_1}{\partial x}- u_1),(\frac{\partial D_2}{\partial y}- u_2) $ respectively over $\Omega$.
The next term is similar to the previous one. Therefore the simplification will be same as above. Hence skipping the calculations we directly put the result as follows: 

\begin{equation}
\begin{split}
P_3 & = \sum_{k=1}^{n_{el}} \tau_3' (\bigtriangledown \cdot \tilde{\bigtriangledown}E^{A,n,\theta}_{c} - \textbf{u} \cdot \bigtriangledown E^{A,n,\theta}_{c}- \alpha E^{A,n,\theta}_{c},\bigtriangledown \cdot \tilde{\bigtriangledown} E_{c}^{A,n,\theta} + \textbf{u} \cdot \bigtriangledown E_{c}^{A,n,\theta} \\
& \quad -  \alpha E_{c}^{A,n,\theta} )_{\Omega_k} \\
& \leq \mid \tau_3' \mid \{ \sum_{k=1}^{n_{el}}(D_{1m}^2 B_{7k}^2 + 2 D_{1m} D_{2m} B_{7k} B_{7k}' + D_{1m} \bar{D}_{u1} B_{8k} B_{7k} + D_{1m} \bar{D}_{u2} B_{8k}' B_{7k} +  \\
& \quad\alpha D_{1m} B_{7k} B_{6k} + D_{2m}^2 B_{7k}'^2 + D_{2m} \bar{D}_{u1} B_{8k} B_{7k}' + D_{2m} \bar{D}_{u2} B_{8k}' B_{7k}' + \alpha D_{2m} B_{6k} B_{7k}' + \\
& \quad  D_{1m} D_{u1} B_{8k} B_{7k}+ D_{2m} D_{u1} B_{8k} B_{7k}' + D_{u1} \bar{D}_{u1} B_{7k}^2 + D_{u1} \bar{D}_{u2} B_{8k} B_{8k}'+ \alpha D_{u1}B_{6k}  \\
& \quad  B_{8k}+ D_{1m} D_{u2} B_{8k}' B_{7k} + D_{2m} D_{u2} B_{8k}' B_{7k}' + D_{u2} \bar{D}_{u1} B_{8k} B_{8k}' + D_{u2} \bar{D}_{u2} B_{8k}'^2 + \alpha \\
& \quad ( D_{u2}  B_{8k}' B_{6k} + D_{1m} B_{6k} B_{7k} +D_{2m} B_{6k} B_{7k}' + \bar{D}_{u1} B_{8k} B_{6k} +  \bar{D}_{u2} B_{8k}' B_{6k} + \alpha B_{6k}^2)) \} \\
& \leq \frac{C_{\tau_3} T}{(T_0-C_{\tau_3})} \sum_{k=1}^{n_{el}} D_{B_{5k}} 
\end{split}
\end{equation}
where $D_{B_{5k}}$ is a notation denoting the big sum of the constants.\vspace{1mm}\\
Now combining all the bounds obtained for $P_1,P_2,P_3$ and putting them into the expression of $I^4_3$ we will have

\begin{equation}
\begin{split}
-I^4_3 & \leq \frac{C_{\tau_3} T}{dt(T_0-C_{\tau_3})}\{ \sum_{k=1}^{n_{el}}(\alpha B_{6k} + D_{B_{1k}})\}(\|E^{A,n+1}_{c}\|^2-\|E^{A,n}_{c}\|^2)+ \frac{C_{\tau_3}Ch^2 }{(T_0-C_{\tau_3})}   \{\sum_{k=1}^{n_{el}} D_{B_{1k}}\} \\
& \quad (\mid c^{n+1} \mid_2 +  \mid c^n \mid_2)+ \frac{C_{\tau_3} T C}{(T_0-C_{\tau_3})}  \{ (\sum_{k=1}^{n_{el}} D_{B_{2k}})(\frac{1+\theta}{2}\mid c^{n+1}\mid_2 +\frac{1-\theta}{2} \mid c^n \mid_2)+ h  \\
& \quad  (\sum_{k=1}^{n_{el}} D_{B_{3k}})(\frac{1+\theta}{2}\mid c^{n+1}\mid_1 +\frac{1-\theta}{2}  \mid c^n \mid_1)+ h^2 (\sum_{k=1}^{n_{el}} D_{B_{4k}})(\frac{1+\theta}{2} \| c^{n+1}\|+ \frac{1-\theta}{2} \| c^n \|) \}  \\
& \quad   + \frac{C_{\tau_3} T}{(T_0-C_{\tau_3})} \sum_{k=1}^{n_{el}} D_{B_{5k}}
\end{split}
\end{equation}
Finally here the process of finding bound for each term of $I_3$ is completed. Now we will focus on finding bounds for the terms of $I_4$. Before going to derivation let us see the term $d_4$ explicitly. \vspace{1mm} \\
\begin{equation}
\begin{split}
d_4 &= \sum_{i=1}^{n+1}(\frac{1}{dt}\tau_3')^i(\partial_t(c^n-c_{h}^n) - \bigtriangledown \cdot \tilde{\bigtriangledown} (c^{n,\theta}-c_{h}^{n,\theta}) + \textbf{u}^n \cdot \bigtriangledown (c^{n,\theta}-c_{h}^{n,\theta}) + \alpha (c^{n,\theta}-c_{h}^{n,\theta})) \\
& = \sum_{i=1}^{n+1}(\frac{1}{dt}\tau_3')^i (\partial_t (E^{I,n}_{c}+ E_{c}^{A,n})- \bigtriangledown \cdot \tilde{\bigtriangledown} (E^{I,n,\theta}_{c}+ E_{c}^{A,n,\theta}) + \textbf{u}^n \cdot \bigtriangledown ((E^{I,n,\theta}_{c}+ E_{c}^{A,n,\theta}) +\\
& \quad +  \alpha (E^{I,n,\theta}_{c}+ E_{c}^{A,n,\theta})) \\
& \leq \sum_{i=1}^{\infty}(\frac{1}{dt} \tau_3' )^i (\partial_t (E^{I,n}_{c}+ E_{c}^{A,n})- \bigtriangledown \cdot \tilde{\bigtriangledown} (E^{I,n,\theta}_{c}+ E_{c}^{A,n,\theta}) + \textbf{u}^n \cdot \bigtriangledown (E^{I,n,\theta}_{c}+ E_{c}^{A,n,\theta}) \\
& \quad +  \alpha (E^{I,n,\theta}_{c}+ E_{c}^{A,n,\theta})) \\
& = \frac{\tau_3' }{(dt- \tau_3' )}(\partial_t E^{I,n}_{c}+\partial_t E^{A,n}_{c})-\frac{ \tau_3' }{(dt- \tau_3')}(\bigtriangledown \cdot \tilde{\bigtriangledown} E^{I,n,\theta}_{c} - \textbf{u}^n \cdot \bigtriangledown E^{I,n,\theta}_{c}-\alpha E^{I,n,\theta}_{c}) \\
& \quad -\frac{\tau_3' }{(dt- \tau_3' )}(\bigtriangledown \cdot \tilde{\bigtriangledown} E^{A,n,\theta}_{c} - \textbf{u}^n \cdot \bigtriangledown E^{A,n,\theta}_{c}-\alpha E^{A,n,\theta}_{c})
\end{split}
\end{equation} 
Since $\frac{\tau_3}{dt+ \tau_3} < 1$, which implies $\frac{\tau_3'}{dt} < 1$ and therefore the series $\sum_{i=1}^{\infty}(\frac{1}{dt} \tau_3' )^i$ converges to $\frac{\tau_3' }{(dt- \tau_3' )}$ 
\begin{equation}
\begin{split}
-I_4 & = \sum_{k=1}^{n_{el}} \tau_3' (d_4, \bigtriangledown \cdot \tilde{\bigtriangledown} E_{c}^{A,n,\theta} +\textbf{u} \cdot \bigtriangledown E_{c}^{A,n,\theta} - \alpha E_{c}^{A,n,\theta} )_{\Omega_k}\\
& \leq \frac{\tau_3'^2}{(dt- \tau_3' )}\sum_{k=1}^{n_{el}}(\partial_t E^{I,n}_{c}+\partial_t E^{A,n}_{c},\bigtriangledown \cdot \tilde{\bigtriangledown} E_{c}^{A,n,\theta} +\textbf{u} \cdot \bigtriangledown E_{c}^{A,n,\theta} - \alpha E_{c}^{A,n,\theta} )_{\Omega_k} \\
& \quad - \frac{\tau_3'^2}{(dt-\tau_3' )} \sum_{k=1}^{n_{el}}(\bigtriangledown \cdot \tilde{\bigtriangledown} E^{I,n,\theta}_{c} - \textbf{u} \cdot \bigtriangledown E^{I,n,\theta}_{c}-\alpha E^{I,n,\theta}_{c},\bigtriangledown \cdot \tilde{\bigtriangledown} E_{c}^{A,n,\theta} +  \\
& \quad \textbf{u} \cdot \bigtriangledown E_{c}^{A,n,\theta}- \alpha E_{c}^{A,n,\theta} )_{\Omega_k}  -\frac{ \tau_3'^2}{(dt- \tau_3')}\sum_{k=1}^{n_{el}}(\bigtriangledown \cdot \tilde{\bigtriangledown} E^{A,n,\theta}_{c} - \textbf{u} \cdot \bigtriangledown E^{A,n,\theta}_{c}- \\
& \quad \alpha E^{A,n,\theta}_{c}, \bigtriangledown \cdot \tilde{\bigtriangledown} E_{c}^{A,n,\theta}+ \textbf{u} \cdot \bigtriangledown E_{c}^{A,n,\theta} - \alpha E_{c}^{A,n,\theta} )_{\Omega_k} \\ 
& \leq \frac{\tau_3^2}{dt (dt+ \tau_3)}\{ \sum_{k=1}^{n_{el}}(\alpha B_{6k} +D_{B_{1k}})\} (\|E^{A,n+1}_{c}\|^2- \|E^{A,n}_{c}\|^2)+ \frac{\tau_3^2 Ch^2}{dt (dt+ \tau_3)}   \\
& \quad \{\sum_{k=1}^{n_{el}} D_{B_{1k}}\} (\mid c^{n+1} \mid_2  + \mid c^n \mid_2)+ \frac{ \tau_3}{dt} C \{ (\sum_{k=1}^{n_{el}} D_{B_{2k}})(\frac{1+\theta}{2}\mid c^{n+1}\mid_2 +\frac{1-\theta}{2}  \\
& \quad \mid c^n \mid_2)+ h (\sum_{k=1}^{n_{el}} D_{B_{2k}})(\frac{1+\theta}{2}\mid c^{n+1}\mid_1 +\frac{1-\theta}{2}  \mid c^n \mid_1) + h^2 (\sum_{k=1}^{n_{el}} D_{B_{4k}})(\frac{1+\theta}{2} \\
& \quad  \| c^{n+1}\|+ \frac{1-\theta}{2} \| c^n \|) \} + \frac{\tau_3}{ dt} \sum_{k=1}^{n_{el}} D_{B_{5k}}\\
\end{split}
\end{equation} 
\begin{equation}
\begin{split}
& \leq \frac{C_{\tau_3}^2}{dt (T_0- C_{\tau_3})}\{ \sum_{k=1}^{n_{el}}(\alpha B_{6k} +D_{B_{1k}})\} (\|E^{A,n+1}_{c}\|^2- \|E^{A,n}_{c}\|^2)+ \frac{C_{\tau_3}^2 Ch^2}{T_0 (T_0- C_{\tau_3})} \\
& \quad \{\sum_{k=1}^{n_{el}} D_{B_{1k}}\} (\mid c^{n+1} \mid_2  + \mid c^n \mid_2)+ \frac{ C_{\tau_3}}{T_0} C \{ (\sum_{k=1}^{n_{el}} D_{B_{2k}})(\frac{1+\theta}{2}\mid c^{n+1}\mid_2 +\frac{1-\theta}{2}\\
& \quad \mid c^n \mid_2)+ h (\sum_{k=1}^{n_{el}} D_{B_{3k}})(\frac{1+\theta}{2}\mid c^{n+1}\mid_1 +\frac{1-\theta}{2}  \mid c^n \mid_1) + h^2 (\sum_{k=1}^{n_{el}} D_{B_{4k}})(\frac{1+\theta}{2} \\
& \quad  \| c^{n+1}\|+ \frac{1-\theta}{2} \| c^n \|) \} + \frac{C_{\tau_3}}{ T_0} \sum_{k=1}^{n_{el}} D_{B_{5k}}\\
\end{split}
\end{equation}
This completes finding the bounds for $I_4$. \vspace{1mm} \\
Now we will find bounds for  $I_5$ and $I_6$ in similar manner as many terms of $I_5, I_6$ coincide with the terms of $I_3$ and $I_4$. \vspace{1mm}\\
\begin{equation}
\begin{split}
- I_5 & =(1-\tau_3^{-1}\tau_3') \sum_{k=1}^{n_{el}} ( \partial_t E^{I,n}_{c} + \partial_t E^{A,n}_{c} - \bigtriangledown \cdot \tilde{\bigtriangledown}(E^{I,n,\theta}_{c}+E_{c}^{A,n,\theta})+ \textbf{u}^n \cdot \bigtriangledown (E^{I,n,\theta}_{c}+ \\
& \quad E_{c}^{A,n,\theta}) + \alpha (E^{I,n,\theta}_{c}+E_{c}^{A,n,\theta}), E_{c}^{A,n,\theta} )_{\Omega_k} \\
& = (1-\tau_3^{-1}\tau_3') \sum_{k=1}^{n_{el}} (\partial_t E^{A,n}_{c}, E_{c}^{A,n,\theta} )_{\Omega_k} + (1-\tau_3^{-1}\tau_3') \sum_{k=1}^{n_{el}}(\textbf{u}^n \cdot \bigtriangledown E^{I,n,\theta}_{c}- \bigtriangledown \cdot \tilde{\bigtriangledown} E^{I,n,\theta}_{c}, \\
& \quad E_{c}^{A,n,\theta} )_{\Omega_k} + (1-\tau_3^{-1}\tau_3') \sum_{k=1}^{n_{el}} (\textbf{u}^n \cdot \bigtriangledown E^{A,n,\theta}_{c}- \bigtriangledown \cdot \tilde{\bigtriangledown} E^{A,n,\theta}_{c} + \alpha E^{A,n,\theta}_{c} ,E^{A,n,\theta}_{c})_{\Omega_k} \\
& = Q_1 + Q_2 + Q_3
\end{split}
\end{equation}
where 
\begin{equation}
\begin{split}
Q_1 & = (1-\tau_3^{-1}\tau_3')\sum_{k=1}^{n_{el}} (\partial_t E_{c}^{A,n}, E_{c}^{A,n,\theta})_{\Omega_k} \\
&= \frac{\tau_3}{dt+\tau_3} \sum_{k=1}^{n_{el}} (\partial_t E_{c}^{A,n}, E_{c}^{A,n,\theta})_{\Omega_k} \\
& \leq \frac{C_{\tau_3}}{dt(T_0-C_{\tau_3})} (\sum_{k=1}^{n_{el}} B_{6k})(\|E_{c}^{A,n+1}\|^2 - \|E_{c}^{A,n}\|^2)
\end{split}
\end{equation}
\begin{equation}
\begin{split}
Q_2 & = (1-\tau_3^{-1}\tau_3')\sum_{k=1}^{n_{el}}(\textbf{u}^n \cdot \bigtriangledown E^{I,n,\theta}_{c}- \bigtriangledown \cdot \tilde{\bigtriangledown} E^{I,n,\theta}_{c}, E_{c}^{A,n,\theta} )_{\Omega_k} \\
& \leq \frac{C_{\tau_3}}{(T_0-C_{\tau_3})} \{\sum_{k=1}^{n_{el}} (\bar{D}_{u1} + \bar{D}_{u2} )B_{6k}\}  \mid E^{I,n,\theta}_{c} \mid_1\\
& \leq \frac{C_{\tau_3} Ch}{(T_0-C_{\tau_3})} \{\sum_{k=1}^{n_{el}} (\bar{D}_{u1} + \bar{D}_{u2} )B_{6k}\} (\frac{1+\theta}{2} \mid c^{n+1} \mid_1 + \frac{1-\theta}{2} \mid c^n \mid_1)
\end{split}
\end{equation}
and 
\begin{equation}
\begin{split}
Q_3 & = (1-\tau_3^{-1}\tau_3') \sum_{k=1}^{n_{el}} (\textbf{u}^n \cdot \bigtriangledown E^{A,n,\theta}_{c}- \bigtriangledown \cdot \tilde{\bigtriangledown} E^{A,n,\theta}_{c} + \alpha E^{A,n,\theta}_{c} ,E^{A,n,\theta}_{c})_{\Omega_k} \\
& \leq \frac{C_{\tau_3}}{(T_0-C_{\tau_3})} \sum_{k=1}^{n_{el}}(D_{1m}B_{7k}+D_{2m}B_{7k}'+\bar{D}_{u1}B_{8k}+\bar{D}_{u2}B_{8k}'+\alpha B_{6k})B_{6k}
\end{split}
\end{equation}
Combining all these results we will have
\begin{equation}
\begin{split}
-I_5 & \leq \frac{C_{\tau_3}}{(T_0-C_{\tau_3})} \{ (\sum_{k=1}^{n_{el}}\frac{B_{6k}}{dt})(\|E_{c}^{A,n+1}\|^2 - \|E_{c}^{A,n}\|^2)+ \\
& \quad  (\sum_{k=1}^{n_{el}} (\bar{D}_{u1} + \bar{D}_{u2} )B_{6k}) Ch (\frac{1+\theta}{2} \mid c^{n+1} \mid_1 + \frac{1-\theta}{2} \mid c^n \mid_1) \\
& \quad +\sum_{k=1}^{n_{el}} (D_{1m}B_{7k}+D_{2m}B_{7k}'+\bar{D}_{u1}B_{8k}+\bar{D}_{u2}B_{8k}'+\alpha B_{6k})B_{6k} \}
\end{split}
\end{equation}
This completes finding the bound for each term of $I_5$. Now we focus on deriving bounds of $I_6$
\begin{equation}
\begin{split}
-I_6 & = -\sum_{k=1}^{n_{el}} \tau_3^{-1}\tau_3' (d_4, E^{A,n,\theta}_{c})_{\Omega_k} \\
& \leq  \frac{\mid \tau_3^{-1} \mid \tau_3'^2}{(dt-\mid \tau_3' \mid)} \{ (\partial_t E^{I,n}_{c}+\partial_t E^{A,n}_{c},E^{A,n,\theta}_{c})_{\Omega_k}-(\bigtriangledown \cdot \tilde{\bigtriangledown} E^{I,n,\theta}_{c} - \textbf{u} \cdot \bigtriangledown E^{I,n,\theta}_{c}-\alpha E^{I,n,\theta}_{c}, \\
& \quad E^{A,n,\theta}_{c})_{\Omega_k} -(\bigtriangledown \cdot \tilde{\bigtriangledown} E^{A,n,\theta}_{c} - \textbf{u} \cdot \bigtriangledown E^{A,n,\theta}_{c}-\alpha E^{A,n,\theta}_{c},E^{A,n,\theta}_{c})_{\Omega_k} \} \\
& = \frac{ \tau_3}{(dt+\tau_3)} \{ (\partial_t E^{A,n}_{c},E^{A,n,\theta}_{c})_{\Omega_k}-(\bigtriangledown \cdot \tilde{\bigtriangledown} E^{I,n,\theta}_{c} - \textbf{u} \cdot \bigtriangledown E^{I,n,\theta}_{c}, E^{A,n,\theta}_{c})_{\Omega_k} - \\
& \quad (\bigtriangledown \cdot \tilde{\bigtriangledown} E^{A,n,\theta}_{c} - \textbf{u} \cdot \bigtriangledown E^{A,n,\theta}_{c}-\alpha E^{A,n,\theta}_{c},E^{A,n,\theta}_{c})_{\Omega_k} \} \\
& \leq \frac{ \tau_3}{(dt+\tau_3)} \{ (\sum_{k=1}^{n_{el}}
\frac{B_{6k}}{dt})(\|E^{A,n+1}_{c}\|^2-\|E^{A,n}_{c}\|^2) + (\sum_{k=1}^{n_{el}} (D_{1m}+D_{2m})B_{6k}) \mid E^{I,n,\theta}_{c}\mid_2 \\
& \quad + (\sum_{k=1}^{n_{el}} (\bar{D}_{u1} + \bar{D}_{u2})B_{6k})\mid E^{I,n,\theta}_{c}\mid_1 +\sum_{k=1}^{n_{el}}(D_{1m} B_{7k} +D_{2m} B_{7k}' + \bar{D}_{u1} B_{8k} + \bar{D}_{u2}B_{8k}' \\
& \quad + \alpha B_{6k} )B_{6k} \} \\
& \leq \frac{ C_{\tau_3}}{(T_0-C_{\tau_3})} \{ (\sum_{k=1}^{n_{el}}
\frac{B_{6k}}{dt})(\|E^{A,n+1}_{c}\|^2-\|E^{A,n}_{c}\|^2)+ (\sum_{k=1}^{n_{el}} C(D_{1m}+D_{2m})B_{6k})  (\frac{1+\theta}{2} \\
& \quad \mid c^{n+1} \mid_2 +\frac{1-\theta}{2} \mid c^n \mid_2)+ C h(\sum_{k=1}^{n_{el}}  (\bar{D}_{u1} + \bar{D}_{u2})B_{6k}) (\frac{1+\theta}{2} \mid c^{n+1} \mid_1 +\frac{1-\theta}{2} \mid c^n \mid_1)\\
\end{split}
\end{equation}
\begin{equation}
\begin{split}
& \quad + \sum_{k=1}^{n_{el}}( D_{1m} B_{7k} +D_{2m} B_{7k}' + \bar{D}_{u1} B_{8k} + \bar{D}_{u2}B_{8k}' + \alpha B_{6k} )B_{6k} \}
\end{split}
\end{equation}
Finally  we have completed finding bounds for each of the terms in the right hand side of (32). Now we explain the further proceeding in language as follows: \vspace{1mm} \\
First we put all the bounds, obtained for each of the terms in the right hand side of (32). Then we  take out few common terms in the left hand side and consequently we have left 3 types of terms in the right hand side. One type will be few constant terms multiplied by $h^2$, other type will be another few constant terms multiplied by h and the remaining constant terms will be free of h. Now we multiply both sides by 2$dt$ and taking summation over n=0,1,...,$(N-1)$ to both the sides. Finally we have (32) as follows:
\begin{multline}
\{1-\frac{2C_{\tau_3}(T+C_{\tau_3})}{T_0-C_{\tau_3}} \sum_{k=1}^{n_{el}}(\alpha B_{6k}+ D_{B_{1k}})- \frac{4C_{\tau_3}}{T_0-C_{\tau_3}} (\sum_{k=1}^{n_{el}} B_{6k})\} \sum_{n=0}^{N-1}  (\|E^{A,n+1}_{c}\|^2-\|E^{A,n}_{c}\|^2) \\
\quad + (\mu_l-\frac{1}{ \epsilon_1}-\frac{\mu_u}{ \epsilon_4}) \sum_{n=0}^{N-1}  (\|\frac{\partial E^{A,n,\theta}_{u1}}{\partial x}\|^2 + \|\frac{\partial E^{A,n,\theta}_{u2}}{\partial y}\|^2 )dt+ (\mu_l-\frac{\mu_u}{ \epsilon_4}) \sum_{n=0}^{N-1} (\|\frac{\partial E^{A,n,\theta}_{u1}}{\partial y}\|^2+\\
\quad  \|\frac{\partial E^{A,n,\theta}_{u2}}{\partial x}\|^2 )dt  + (D_l-\frac{D_m}{ \epsilon_2}-\frac{C_1^n}{ \epsilon_3})\sum_{n=0}^{N-1} \|\frac{\partial E^{A,n,\theta}_{c}}{\partial x}\|^2 dt + (D_l-\frac{D_m}{ \epsilon_2}-\frac{C_2^n}{ \epsilon_3})\sum_{n=0}^{N-1} \| \frac{\partial E^{A,n,\theta}_{c}}{\partial y}\|^2 dt\\
\quad   + 
 \sigma \sum_{n=0}^{N-1} (\|E^{A,n,\theta}_{u1}\|^2 + \|E^{A,n,\theta}_{u2}\|^2) dt +  (\alpha- 2\epsilon_3(C_1^n + C_2^n) \|E^{A,n,\theta}_{c}\|^2 dt \\
\leq 2C h^2 \sum_{n=0}^{N-1} [C \mu_u \epsilon_4 \sum_{i=1}^2 (\frac{1+\theta}{2} \mid u_i^{n+1} \mid_2 + \frac{1-\theta}{2} \mid u_i^n \mid_2)^2 + C \epsilon_1 (\frac{1+\theta}{2} \mid p^{n+1} \mid_1 + \frac{1-\theta}{2} \mid p^n \mid_1)^2 + \\
\quad  C(\frac{D_m \epsilon_2}{2} + \frac{C_1^n+C_2^n}{2 \epsilon_3}) (\frac{1+\theta}{2} \mid c^{n+1} \mid_2 + \frac{1-\theta}{2} \mid c^n \mid_2)^2 + \frac{C_{\tau_3}(T_0+C_{\tau_3})}{T_0(T_0-C_{\tau_3})} (\sum_{k=1}^{n_{el}} D_{B_{1k}})(\mid c^{n+1} \mid_2 + \\
\quad    \mid c^n \mid_2)+ C_{\tau_1} (\sum_{k=1}^{n_{el}} \sigma (\mu_u B_{2k} + \sigma B_{1k} + B_{3k}))(\frac{1+\theta}{2} \mid u_1^{n+1} \mid_2 + \frac{1-\theta}{2} \mid u_1^{n} \mid_2)+ C_{\tau_1} (\sum_{k=1}^{n_{el}} \sigma (\mu_u B_{2k}' + \\
\quad   \sigma  B_{1k}' + B_{3k}'))(\frac{1+\theta}{2} \mid u_2^{n+1} \mid_2 + \frac{1-\theta}{2} \mid u_2^{n} \mid_2)+ (\sum_{k=1}^{n_{el}} D_{B_{4k}})\frac{C_{\tau_3}}{T_0}(\frac{1+\theta}{2} \| c^{n+1} \|+ \frac{1-\theta}{2} \| c^{n} \|) ] dt \\
\quad   + 2 C h \sum_{n=0}^{N-1}[ C_{\tau_1} (\sum_{k=1}^{n_{el}}(\mu_u(B_{2k}+ B_{2k}')+ \sigma (B_{1k}+B_{1k}')+ (B_{3k}+B_{3k}'))(\frac{1+\theta}{2}\mid p^{n+1} \mid_1 + \frac{1-\theta}{2} \mid p^n \mid_1) \\
\quad   + C_{\tau_2} (\sum_{k=1}^{n_{el}}B_{4k}) (\frac{1+\theta}{2} \mid u_1^{n+1} \mid_1 + \frac{1-\theta}{2} \mid u_1^n \mid_1) + C_{\tau_2} (\sum_{k=1}^{n_{el}}B_{5k}')(\frac{1+\theta}{2} \mid u_2^{n+1} \mid_1 +\frac{1-\theta}{2} \mid u_2^n \mid_1) \\
\end{multline}
\begin{multline}
\quad + \{ (\frac{C_{\tau_3} T }{(T_0-C_{\tau_3})}+\frac{C_{\tau_3}}{T_0})(\sum_{k=1}^{n_{el}} D_{B_{3k}})+\frac{4C_{\tau_3}}{T_0-C_{\tau_3}} (\sum_{k=1}^{n_{el}}(\bar{D}_{u1}+\bar{D}_{u2})B_{6k})\}(\frac{1+\theta}{2} \mid c^{n+1} \mid_1 +\frac{1-\theta}{2} \\
\quad \mid c^n \mid_1)] dt +2C \sum_{n=0}^{N-1} \{ C_{\tau_1} (\sum_{k=1}^{n_{el}}(\mu_u^2 B_{2k} + \sigma \mu_u B_{1k} +\mu_u B_{3k} ))(\frac{1+\theta}{2} \mid u_1^{n+1} \mid_2 +\frac{1-\theta}{2}\mid u_1^n \mid_2)+\\
\quad C_{\tau_1} (\sum_{k=1}^{n_{el}}(\mu_u^2 B_{2k}' + \sigma \mu_u B_{1k}' +\mu_u B_{3k}' ))(\frac{1+\theta}{2} \mid u_2^{n+1} \mid_2 +\frac{1-\theta}{2}\mid u_2^n \mid_2)+ (\frac{C_{\tau_3} T }{(T_0-C_{\tau_3})}+\frac{C_{\tau_3}}{T_0}) \\  
\quad (\sum_{k=1}^{n_{el}}D_{B_{2k}})(\frac{1+\theta}{2} \mid c^{n+1} \mid_1 +\frac{1-\theta}{2}\mid c^n \mid_1) + C_{\tau_1} (\sum_{k=1}^{n_{el}}(M_{1k}+M_{2k}))+(\frac{C_{\tau_3} T }{(T_0-C_{\tau_3})}+\frac{C_{\tau_3}}{T_0}) \\
\quad  (\sum_{k=1}^{n_{el}} D_{B_{5k}}) +C_{\tau_2} (\sum_{k=1}^{n_{el}}(B_{4k}^2+ B_{5k}'^2))+(\sum_{k=1}^{n_{el}} (D_{1m}+D_{2m})B_{6k})  (\frac{1+\theta}{2}\mid c^{n+1} \mid_2 +\frac{1-\theta}{2} \mid c^n \mid_2)\\
\quad + 2( \sum_{k=1}^{n_{el}}( D_{1m} B_{7k} +D_{2m} B_{7k}' + \bar{D}_{u1} B_{8k} + \bar{D}_{u2}B_{8k}' + \alpha B_{6k} )B_{6k})\} dt \hspace{10mm}
\end{multline}
We can choose the values of the elements in such a manner that we can make all the coefficients in the left hand side positive. Now after taking minimum of all the coefficients in left hand side, let us divide both the sides with that minimum, which turns out to be a positive real number. Now by applying initial condition on $c$ we will have $\|E^{A,0}_{c}\|=0$. \vspace{1mm}\\
After performing all these intermediate steps we will finally arrive at the following expression:
\begin{multline}
\|E^{A,N}_{c}\|^2 +\sum_{n=0}^{N-1}  (\|\frac{\partial E^{A,n,\theta}_{u1}}{\partial x}\|^2 + \|\frac{\partial E^{A,n,\theta}_{u1}}{\partial y}\|^2 + \|E^{A,n,\theta}_{u1} \|^2 )dt+ \\
 \quad \sum_{n=0}^{N-1}  (\|\frac{\partial E^{A,n,\theta}_{u2}}{\partial x}\|^2 + \|\frac{\partial E^{A,n,\theta}_{u2}}{\partial y}\|^2 + \|E^{A,n,\theta}_{u2} \|^2 )dt + \\
 \quad \sum_{n=0}^{N-1}  (\|\frac{\partial E^{A,n,\theta}_{c}}{\partial x}\|^2 + \|\frac{\partial E^{A,n,\theta}_{c}}{\partial y}\|^2 + \|E^{A,n,\theta}_{c} \|^2 )dt \leq C(T,\textbf{u},p,c) (h^2+h+dt^{2r})
\end{multline}
This implies
\begin{equation}
\|E^{A,N}_{u1}\|_{l^2(H^1)}^2+ \|E^{A,N}_{u2}\|_{l^2(H^1)}^2+ \|E^{A,N}_{c}\|_V^2 \leq C(T,\textbf{u},p,c) (h^2+h+dt^{2r})
\end{equation}
where
\begin{equation}
    r=
    \begin{cases}
      1, & \text{if}\ \theta=1 \\
      2, & \text{if}\ \theta=0
    \end{cases}
  \end{equation}
We have used the fact that $\sum_{n=0}^{N-1} g_n dt \leq C T \sum_{n=0}^{N-1} g_n $. This completes the first part of the proof. \vspace{2mm}\\
\textbf{Second part} Using this above result we are going to estimate auxiliary error part of pressure. We will use inf-sup condition to find estimate for $E_p^A$. Applying Galerkin orthogonality only for variational form of Stokes-Darcy flow problem we have obtained \\
\begin{equation}
\begin{split}
 a_S (\textbf{u}-\textbf{u}_h, \textbf{v}_h)- b(\textbf{v}_h, p-p_h) & =0 \\ 
b(\textbf{v}_h, p-I_hp)+ b(\textbf{v}_h, I_hp-p_h) & =  a_S(E^I_\textbf{u}, \textbf{v}_h)+  a_S(E^A_\textbf{u}, \textbf{v}_h)
\end{split}
\end{equation}
Using the inclusion $\bigtriangledown \cdot V_s^h \subset Q_s^h$ and the property of the $L^2$ orthogonal projection of $I^h_p$ we have 
\begin{equation}
b(\textbf{v}_h, p-I_hp)= \int_{\Omega}(p-I_hp)(\bigtriangledown \cdot \textbf{v})=0
\end{equation}
Now according to inf-sup condition we will have the following expression
\begin{equation}
\begin{split}
\|p-I_hp\|_Q^2& = \|E_p^A\|_Q^2 \\
& = \sum_{n=0}^{N-1} \|E_p^{A,n,\theta}\|^2 dt \\
& \leq  \sum_{n=0}^{N-1} \underset{\textbf{v}_h}{sup} \frac{b(\textbf{v}_h, E_p^{A,n,\theta})}{\|\textbf{v}_h\|_1} dt
\end{split}
\end{equation}
Now from (72)
\begin{equation}
\begin{split}
\sum_{n=0}^{N-1} b(\textbf{v}_h, E_p^{A,n,\theta}) dt & = \sum_{n=0}^{N-1}\{ a_S(E^{I,n,\theta}_\textbf{u}, \textbf{v}_h)+a_S(E^{A,n,\theta}_\textbf{u}, \textbf{v}_h)\} dt\\
& = \sum_{n=0}^{N-1} \{ \int_{\Omega} \mu(c^n) \bigtriangledown E^{I,n,\theta}_{u1} \cdot \bigtriangledown v_{1h} +  \int_{\Omega} \mu(c^n) \bigtriangledown E^{I,n,\theta}_{u2} \cdot \bigtriangledown v_{2h}+ \\
& \quad \sigma \int_{\Omega}  (E^{I,n,\theta}_{u1} v_{1h} +E^{I,n,\theta}_{u2} v_{2h})+  \int_{\Omega} \mu(c^n) \bigtriangledown E^{A,n,\theta}_{u1} \cdot \bigtriangledown v_{1h} + \\
& \quad \int_{\Omega} \mu(c^n) \bigtriangledown E^{A,n,\theta}_{u2} \cdot \bigtriangledown v_{2h} +  \sigma \int_{\Omega}  (E^{A,n,\theta}_{u1} v_{1h} +E^{A,n,\theta}_{u2} v_{2h}) \} dt\\
& \leq (\mu_u + \sigma) \sum_{n=0}^{N-1} \{(\|E^{I,n,\theta}_{u1}\|_1+\|E^{A,n,\theta}_{u1}\|_1) \|v_{1h}\|_1 + (\|E^{I,n,\theta}_{u2}\|_1 + \\
& \quad \|E^{A,n,\theta}_{u2}\|_1) \|v_{2h}\|_1 \} dt \\
& \leq (\mu_u + \sigma) \sum_{n=0}^{N-1} \{(\|E^{I,n,\theta}_{u1}\|_1 + \|E^{I,n,\theta}_{u2}\|_1 +\|E^{A,n,\theta}_{u1}\|_1+\|E^{A,n,\theta}_{u2}\|_1)\\
& \quad (\|v_{1h}\|_1+\|v_{2h}\|_1)\}dt\\
\end{split}
\end{equation}
\begin{equation}
\begin{split}
& \leq (\mu_u + \sigma)(\|E^{A}_{u1}\|_{l^2(H^1)} + \|E^{A}_{u2}\|_{l^2(H^1)})(\|v_{1h}\|_1+\|v_{2h}\|_1)+ (\mu_u + \sigma)\\
& \quad \sum_{n=0}^{N-1}(\|E^{I,n,\theta}_{u1}\|_1+\|E^{I,n,\theta}_{u2}\|_1)(\|v_{1h}\|_1+\|v_{2h}\|_1)dt\\
& \leq (\mu_u + \sigma)(\|E^{A}_{u1}\|_{l^2(H^1)}^2 + \|E^{A}_{u2}\|_{l^2(H^1)}^2+ \|E^{A}_{c}\|_V^2) \|\textbf{v}_h\|_1 + (\mu_u + \sigma)Ch \\
& \quad \sum_{n=0}^{N-1}\{\sum_{i=1}^{2}(\frac{1+\theta}{2}\mid u_i^{n+1}\mid_2 + \frac{1-\theta}{2}\mid u_i^{n}\mid_2)dt\} \|\textbf{v}_h\|_1 \\
& \leq C'(T,\textbf{u},p,c) (h^2+h+dt^{2r}) \|\textbf{v}_h\|_1
\end{split}
\end{equation}
Using this above result into (74), we will have the estimate for the pressure term
\begin{equation}
\|p-I_hp\|_Q^2 \leq C'(T,\textbf{u},p,c) (h^2+h+dt^{2r})
\end{equation}
Now combining the results obtained in the first and second part we have finally arrived at the auxiliary error estimate as follows
\begin{equation}
\|E^A_{u1}\|^2_{l^2(H^1)} + \|E^A_{u2}\|^2_{l^2(H^1)}+ \|E^A_p\|_{l^2(L^2)}^2  + \|E^A_{c}\|^2_V \leq \bar{C}(T,\textbf{u},p,c) (h^2+h+ dt^{2r})
\end{equation}
where
\begin{equation}
    r=
    \begin{cases}
      1, & \text{if}\ \theta=1 \\
      2, & \text{if}\ \theta=0
    \end{cases}
  \end{equation}
This completes the proof.
\end{proof}
\begin{theorem}(Apriori error estimate)
Assuming the same condition as in the previous theorem, 
\begin{equation}
\|u_1-u_{1h}\|_{l^2(H^1)}^2+ \|u_2-u_{2h}\|_{l^2(H^1)}^2+\|p-p_h\|_{l^2(L^2)}^2 + \|c-c_h\|^2_V \leq C'' (h^2+h+ dt^{2r})
\end{equation}
where C' depends on T, $\textbf{u}$,p,c and
\begin{equation}
    r=
    \begin{cases}
      1, & \text{if}\ \theta=1 \\
      2, & \text{if}\ \theta=0
    \end{cases}
  \end{equation}
\end{theorem}
\begin{proof}
By applying triangle inequality, the interpolation inequalities and the result of the previous theorem we will have,
\begin{multline}
\|u_1-u_{1h}\|_{l^2(H^1)}^2+ \|u_2-u_{2h}\|_{l^2(H^1)}^2 +\|p-p_h\|_{l^2(L^2)}^2 +\|c-c_h\|^2_V \\
 = \|E^I_{u1}+E^A_{u1}\|_{l^2(H^1)}^2 + \|E^I_{u2}+E^A_{u2}\|_{l^2(H^1)}^2 +  \|E^I_{p}+E^A_{p}\|_{l^2(L^2)}^2+ \|E^I_{c}+E^A_{c}\|_V^2 \\
 \leq \bar{C} (\|E^I_{u1}\|_{l^2(H^1)}^2 + \|E^I_{u2}\|_{l^2(H^1)}^2+\|E^I_{c}\|_{l^2(L^2)}^2 +  \|E^I_{c}\|_V^2+ 
 \|E^A_{u1}\|_{l^2(H^1)}^2+\|E^A_{u2}\|_{l^2(H^1)}^2\\
  \quad + \|E^A_{p}\|_{l^2(L^2)}^2+\|E^A_{c}\|_V^2) \hspace{80mm}\\
 \leq C''(T,\textbf{u},p,c)(h^2+h+ dt^{2r}) \hspace{70mm}
\end{multline}
This completes apriori error estimation.
\end{proof}

\subsection{Aposteriori error estimation}
In this section we are going to derive residual based aposteriori error estimation. \\
We have $B(\textbf{V},\textbf{V})= a_S(\textbf{v},\textbf{v})+a_T(d,d) \geq \mu_l (\|v_1\|_1^2+\|v_2\|_1^2)+ D_{\alpha} \|d\|_1^2$ \hspace{1mm} $\forall \textbf{V} \in \textbf{V}_F$ \\
Now we substitute the errors $e_{u1},e_{u2},e_c$ into the relation we will similarly have
\begin{multline}
\mu_l (\|e_{u1}\|_1^2+\|e_{u2}\|_1^2)+ D_{\alpha} \|e_c\|_1^2  \leq a_S(e_\textbf{u},e_\textbf{u})+a_T(e_c,e_c)\\
\end{multline} 
By adding few terms in both sides the above equation becomes
\begin{multline}
\underbrace{(\frac{\partial e_{c}}{\partial t},e_{c})+\mu_l (\|e_{u1}\|_1^2+\|e_{u2}\|_1^2)+\sigma \|e_p\|^2+ D_{\alpha} \|e_c\|_1^2}_\textit{LHS} \\
 \leq \underbrace{(\frac{\partial e_{c}}{\partial t},e_{c})+a_S(e_\textbf{u},e_\textbf{u})+a_T(e_c,e_c)+ (e_p,e_p)+
 b(e_\textbf{u},e_p)-b(e_\textbf{u},e_p)}_\textit{RHS}
\end{multline} 
Now first we will find a lower bound of $LHS$ and then upper bound for $RHS$ and finally combining them we will get aposteriori error estimate. To find the lower bound the $LHS$ can be written as
\begin{multline}
LHS =  (\frac{e_{c}^{n+1}-e_{c}^n}{dt},e_{c}^{n,\theta})+ 
\mu_l \sum_{i=1}^{2}(\|e_{ui}^{n,\theta}\|^2+ \|\frac{\partial e_{ui}^{n,\theta}}{\partial x}\|^2 +  \|\frac{\partial e_{ui}^{n,\theta}}{\partial y}\|^2)+ \\     \sigma \|e_p^{n,\theta}\|^2 + D_{\alpha}(\|e_{c}^{n,\theta}\|^2+
 \|\frac{\partial e_{c}^{n,\theta}}{\partial x}\|^2 + \|\frac{\partial e_{c}^{n,\theta}}{\partial y}\|^2) \hspace{20mm}
\end{multline}
Using the same argument done in (33) we have
\begin{equation}
\begin{split}
(\frac{e_{c}^{n+1}-e_{c}^n}{dt},e_{c}^{n,\theta}) & \geq \frac{1}{2 dt}(\|e_{c}^{n+1}\|^2-\|e_{c}^n\|^2) \\
\end{split}
\end{equation}
Hence
\begin{multline}
\frac{1}{2 dt}(\|e_{c}^{n+1}\|^2-\|e_{c}^n\|^2)+\mu_l \sum_{i=1}^{2}(\|e_{ui}^{n,\theta}\|^2+ \|\frac{\partial e_{ui}^{n,\theta}}{\partial x}\|^2 + \|\frac{\partial e_{ui}^{n,\theta}}{\partial y}\|^2)+ \\
\sigma \|e_p^{n,\theta}\|^2+
 D_{\alpha}(\|e_{c}^{n,\theta}\|^2+
 \|\frac{\partial e_{c}^{n,\theta}}{\partial x}\|^2 + \|\frac{\partial e_{c}^{n,\theta}}{\partial y}\|^2) \leq LHS \leq RHS
\end{multline}
Now our aim is to find upper bound for $RHS$. Before proceeding let us write few notions clearly.\\
$E^A_{u1}=(I^h_{u1} u_1-u_{1h})= (I^h_{u1} u_1-I^h_{u1} u_{1h})=I^h_{u1} e_{u1}$ \vspace{1mm}  \\
Similarly $E^A_{u2}= I^h_{u2} e_{u2}$, \hspace{1mm} $E^A_{p}= I^h_{p} e_{p}$ and $E^A_{c}= I^h_{c} e_{c}$ \vspace{1mm} \\
Now we divide $RHS$ into two broad parts by adding and subtracting corresponding auxiliary errors with each of the terms as follows:

\begin{multline}
RHS = \{ (\frac{\partial e_{c}}{\partial t},e_{c}-I^h_{c} e_{c})+ a_S(e_\textbf{u},e_\textbf{u}-I^h_\textbf{u} e_\textbf{u})+a_T(e_c,e_c-I^h_{c} e_{c})\\
-  b(e_\textbf{u}-I^h_\textbf{u} e_\textbf{u},e_p)+ 
b(e_\textbf{u},e_p-I^h_{p} e_{p})\} + 
\{ (\frac{\partial e_{c}}{\partial t},I^h_{c} e_{c})\\
 + a_S(e_\textbf{u},I^h_\textbf{u} e_\textbf{u})
+a_T(e_c,I^h_{c} e_{c})-  b(I^h_\textbf{u} e_\textbf{u},e_p)+ b(e_\textbf{u},I^h_{p} e_{p})\}
\end{multline}
In the expression of $RHS$ the first under brace part is first part and second one is second part. Before proceeding further let us introduce the residuals corresponding to each equations \\
\[
\textbf{R}^h=
  \begin{bmatrix}
 \textbf{f}_1-( - \mu(c) \Delta \textbf{u}_h + \sigma \textbf{u}_h + \bigtriangledown p_h) \\
  f_2  -\bigtriangledown \cdot \textbf{u}_h \\
 g-(\frac{\partial c_h}{\partial t} - \bigtriangledown \cdot \tilde{\bigtriangledown} c_h + \textbf{u} \cdot \bigtriangledown c_h + \alpha c_h )
  \end{bmatrix}
\]
This column vector $\textbf{R}^h$ has four components $R_1^h, R_2^h, R_3^h$ and $R_4^h$ denoting four rows respectively. Let us start finding bound for the first part as follows: for all $\textbf{v}=(v_1,v_2) \in V_s \times V_s$
\begin{multline}
 \int_{\Omega} \mu(c^n)\bigtriangledown e_\textbf{u}^{n,\theta}: \bigtriangledown \textbf{v} + \int_{\Omega} \sigma e_\textbf{u}^{n,\theta} \cdot \textbf{v} - \int_{\Omega} (\bigtriangledown \cdot \textbf{v}) e_p^{n,\theta} \\
= \{ \int_{\Omega} \mu(c^n)\bigtriangledown \textbf{u}^{n,\theta}: \bigtriangledown \textbf{v} + \int_{\Omega} \sigma \textbf{u}^{n,\theta} \cdot \textbf{v} - \int_{\Omega} (\bigtriangledown \cdot \textbf{v}) p^{n,\theta} \} -\{ \int_{\Omega} \mu(c^n)\bigtriangledown \textbf{u}_h^{n,\theta}: \bigtriangledown \textbf{v} + \hspace{2mm}\\
  \int_{\Omega} \sigma \textbf{u}_h^{n,\theta} \cdot \textbf{v} - \int_{\Omega} (\bigtriangledown \cdot \textbf{v}) p_h^{n,\theta} \} \hspace{50mm}\\
= \int_{\Omega}( - \mu(c^n) \Delta \textbf{u}^{n,\theta} + \sigma \textbf{u}^{n,\theta} + \bigtriangledown p^{n,\theta}) \cdot \textbf{v} - \int_{\Omega}( - \mu(c^n) \Delta \textbf{u}_h^{n,\theta} + \sigma \textbf{u}_h^{n,\theta} + \bigtriangledown p_h^{n,\theta}) \cdot \textbf{v} \hspace{10mm}\\
= (R_1^{h,n,\theta},v_1)+ (R_2^{h,n,\theta},v_2) \hspace{80 mm}
\end{multline}
Similarly $\int_{\Omega} (\bigtriangledown \cdot e_\textbf{u}^{n,\theta})q = \int_{\Omega} R_3^{h,n,\theta} q$ \hspace{2mm} $\forall q \in Q_s$
\begin{multline}
\int_{\Omega} (\frac{ e_c^{n+1}-e_c^{n}}{dt} d + \tilde{\bigtriangledown} e_c^{n,\theta} \cdot \bigtriangledown d + d \textbf{u} \cdot \bigtriangledown e_c^{n,\theta} + \alpha e_c^{n,\theta} d)= \int_{\Omega} R_4^{h,n,\theta} d \hspace{3 mm} \forall d \in V_s \hspace{30 mm}
\end{multline}
Now substituting $v_1, v_2 ,q,d$ in the above expressions by $(e_{u1}-I^h_{u1}e_{u1}),(e_{u2}-I^h_{u2}e_{u2}), (e_{p}-I^h_{p}e_{p}), (e_{c}-I^h_{c}e_{c})$ respectively, we will have the first part of the $RHS$ as,

\begin{equation}
\begin{split}
First \hspace{1mm} part \hspace{1mm} of \hspace{1mm} RHS & = \int_{\Omega}\{ R_1^{h,n,\theta} (e_{u1}^{n,\theta}-I^h_{u1}e_{u1}^{n,\theta}) +  R_2^{h,n,\theta} (e_{u2}^{n,\theta}-I^h_{u2}e_{u2}^{n,\theta}) + R_3^{h,n,\theta} (e_{p}^{n,\theta}-I^h_{p}e_{p}^{n,\theta}) +\\ 
& \quad R_4^{h,n,\theta} (e_{c}^{n,\theta}-I^h_{c}e_{c}^{n,\theta})\}\\
& \leq \|R_1^{h,n,\theta}\| \|e_{u1}^{n,\theta}-I^h_{u1}e_{u1}^{n,\theta}\| + \|R_2^{h,n,\theta}\| \|e_{u2}^{n,\theta}-I^h_{u2}e_{u2}^{n,\theta}\| + \|R_3^{h,n,\theta}\| \|e_{p}^{n,\theta}-I^h_{p}e_{p}^{n,\theta}\| \\
& \quad + \|R_4^{h,n,\theta}\| \|e_{c}^{n,\theta}-I^h_{c}e_{c}^{n,\theta}\| \hspace{10mm}( by \hspace{1mm} Cauchy-Schwarz \hspace{1mm} inequality)\\
& \leq \|R_1^{h,n,\theta}\| h \mid e_{u1}^{n,\theta}\mid_1 + \|R_2^{h,n,\theta}\| h \mid e_{u2}^{n,\theta}\mid_1 + \|R_3^{h,n,\theta}\| h \mid e_{p}^{n,\theta}\mid_1 + \|R_4^{h,n,\theta}\| h \mid e_{c}^{n,\theta}\mid_1\\
& \leq \frac{h^2}{2 \epsilon_1}\|R_1^{h,n,\theta}\|^2+ \frac{\epsilon_1}{2} \mid e_{u1}^{n,\theta}\mid_1^2+ \frac{h^2}{2 \epsilon_1}\|R_2^{h,n,\theta}\|^2+ \frac{\epsilon_1}{2} \mid e_{u2}^{n,\theta}\mid_1^2+ \frac{h^2}{2 \epsilon_1}\|R_3^{h,n,\theta}\|^2+ \\
& \quad  \frac{\epsilon_1}{2} \mid e_{p}^{n,\theta}\mid_1^2 + \frac{h^2}{2 \epsilon_1}\|R_4^{h,n,\theta}\|^2+ \frac{\epsilon_1}{2} \mid e_{c}^{n,\theta}\mid_1^2 \hspace{5mm}( by \hspace{1mm} Young's \hspace{1mm} inequality) \\
& \leq \frac{h^2}{2 \epsilon_1} (\|R_1^{h,n,\theta}\|^2+\|R_2^{h,n,\theta}\|^2+\|R_3^{h,n,\theta}\|^2+\|R_4^{h,n,\theta}\|^2)+ \frac{\epsilon_1}{2} (\|e_{u1}^{n,\theta}\|_1^2+\|e_{u2}^{n,\theta}\|_1^2+\\
& \quad \|e_p^{n,\theta}\|_1^2+ \|e_c^{n,\theta}\|_1^2)
\end{split}
\end{equation}
This completes finding bound for first part of $RHS$. Now we are going to estimate remaining second part of $RHS$. For that we will use subgrid formulation (8). Subtracting (8) from the variational finite element formulation satisfied by the exact solution we have $\forall \textbf{V}_h \in \textbf{V}_F^h $
\begin{multline}
\int_{\Omega} \frac{e_{c}^{n+1}-e_c^n}{ dt} d_{h}+ \int_{\Omega} \mu(c^n)\bigtriangledown e_\textbf{u}^{n,\theta}: \bigtriangledown \textbf{v}_h + \int_{\Omega} \sigma e_\textbf{u}^{n,\theta} \cdot \textbf{v}_h - \int_{\Omega} (\bigtriangledown \cdot \textbf{v}_h) e_p^{n,\theta} + \int_{\Omega} (\bigtriangledown \cdot e_\textbf{u}^{n,\theta}) q_h + \\
  \int_{\Omega} \tilde{\bigtriangledown} e_c^{n,\theta} \cdot \bigtriangledown d_h + \int_{\Omega} d_h \textbf{u} \cdot \bigtriangledown e_c^{n,\theta} + \int_{\Omega} \alpha e_c^{n,\theta} d_h \\
 = \sum_{k=1}^{n_{el}} \{(\tau_k'(\textbf{R}^{h,n,\theta}+\textbf{d}), -\mathcal{L}^* \textbf{V}_h)_{\Omega_k} - ((I-\tau_k^{-1}\tau_k) \textbf{R}^{h,n,\theta}, \textbf{V}_h)_{\Omega_k} + (\tau_k^{-1}\tau_k \textbf{d}, \textbf{V}_h)_{\Omega_k} \} \\
= \sum_{k=1}^{n_{el}} \{ \tau_1'(R_1^{h,n,\theta}, \mu(c) \Delta v_{1h}-\sigma v_{1h}+ \frac{\partial q_h}{\partial x})_{k} +\tau_1'(R_2^{h,n,\theta}, \mu(c) \Delta v_{2h}-\sigma v_{2h}+ \frac{\partial q_h}{\partial y})_{k}+ \\
\quad \tau_2'(R_3^{h,n,\theta}, \bigtriangledown \cdot \textbf{v}_h)_{k} + \tau_3'(R_4^{h,n,\theta}+d_4, \bigtriangledown \cdot \tilde{\bigtriangledown} d_h + \textbf{u} \cdot \bigtriangledown d_h - \alpha d_h)_{k} \\
\quad + (1-\tau_3^{-1}\tau_3')(R_4^{h,n,\theta},d_{h})_k +\tau_3^{-1}\tau_3' (d_4,d_{h})_k\} \hspace{30mm}
\end{multline}
Here $(\cdot,\cdot)_k $ in simple form denotes $(\cdot,\cdot)_{\Omega_k}$ \vspace{1mm}\\
Now substituting $\textbf{V}_h$ by $(I^h_{u1}e_{u1}, I^h_{u2}e_{u2}, I^h_{p}e_{p}, I^h_{c}e_{c})$ in the above equation we will get the second part of $RHS$ as follows

\begin{multline}
 (\frac{e_{c}^{n+1}-e_c^n}{dt},I^h_{c} e_{c})
 + a_S(e_\textbf{u}^{n,\theta},I^h_\textbf{u} e_\textbf{u})
+a_T(e_c^{n,\theta},I^h_{c} e_{c})-b(I^h_\textbf{u} e_\textbf{u},e_p^{n,\theta})+ b(e_\textbf{u}^{n,\theta},I^h_{p} e_{p}) \\
 = \sum_{k=1}^{n_{el}} \{ \tau_1'(R_1^{h,n,\theta}, \mu(c) \Delta I^h_{u1} e_{u1}-\sigma I^h_{u1} e_{u1}+ \frac{\partial I^h_{p} e_{p}}{\partial x})_k + \tau_2'(R_3^{h,n,\theta}, \bigtriangledown \cdot I^h_{\textbf{u}} e_\textbf{u})_k +\\
\quad \tau_1'(R_2^{h,n,\theta},\mu(c) \Delta I^h_{u2} e_{u2} -\sigma I^h_{u2} e_{u2}+ \frac{\partial I^h_{p} e_{p}}{\partial y})_k+ (1-\tau_3^{-1}\tau_3')(R_4^{h,n,\theta},I^h_{c} e_{c})_k + \\
\quad \tau_3'(R_4^{h,n,\theta}+d_4, \bigtriangledown \cdot \tilde{\bigtriangledown} I^h_{c} e_{c} + \textbf{u} \cdot \bigtriangledown I^h_{c} e_{c} - \alpha I^h_{c} e_{c})_k + 
\tau_3^{-1}\tau_3' (d_4,I^h_{c} e_{c})_k\}
\end{multline}
Now we will bound each of the term starting with 4th term of the right hand side of the above equation.
\begin{equation}
\begin{split}
\sum_{k=1}^{n_{el}}(1-\tau_3^{-1}\tau_3')(R_4^{h,n,\theta},I^h_{c} e_{c})_k & \leq \frac{C_{\tau_3}}{T_0-C_{\tau_3}} (\sum_{k=1}^{n_{el}}  B_{6k}) \|R_4^{h,n,\theta}\|
\end{split}
\end{equation}
We have obtained this using Cauchy-Schwarz inequality and then imposing bound on auxiliary error corresponding to $u_1$ over each sub-domain $\Omega_k$. Before proceeding further let us look into the form of the column vector $\textbf{d}$ which has components $d_1,d_2,d_3$ and $d_4$ \vspace{1mm} \\
$\textbf{d}$= $\sum_{i=1}^{n+1}(\frac{1}{dt}M\tau_k')^i(\textbf{F} -M\partial_t \textbf{U}_h - \mathcal{L}\textbf{U}_h)=\sum_{i=1}^{n+1}(\frac{1}{dt}M\tau_k')^i \textbf{R}^h$ \vspace{2mm}\\
Hence clearly $d_1 =0$, $d_2=0$, $d_3= 0$ and $d_4=(\sum_{i=1}^{n+1}(\frac{1}{dt} \tau_3')^i) R_4^{h,n,\theta}$ \vspace{2mm} \\
Now we can bound the last term as follows
\begin{equation}
\begin{split}
\sum_{k=1}^{n_{el}}\tau_3^{-1}\tau_3' (d_4,I^h_{c} e_{c})_k & = \tau_3^{-1}\tau_3' (\sum_{i=1}^{n+1}(\frac{1}{dt} \tau_3')^i) \sum_{k=1}^{n_{el}} (R_4^{h,n,\theta},I^h_{c} e_{c})_k \\
& \leq \tau_3^{-1}\tau_3' (\sum_{i=1}^{\infty}(\frac{1}{dt} \tau_3')^i) \sum_{k=1}^{n_{el}} (R_4^{h,n,\theta},I^h_{c} e_{c})_k \\
& \leq \frac{C_{\tau_3}}{(T_0-C_{\tau_3}} (\sum_{k=1}^{n_{el}} B_{6k}) \|R_4^{h,n,\theta}\|
\end{split}
\end{equation}
Now it will be easy enough to bound the remaining terms of the right hand side.
\begin{equation}
\begin{split}
\sum_{k=1}^{n_{el}} \tau_1'(R_1^{h,n,\theta}, \mu(c) \Delta I^h_{u1} e_{u1}-\sigma I^h_{u1} e_{u1}+ \frac{\partial I^h_{p} e_{p}}{\partial x})_k  
& \leq \mid \tau_1'\mid (\sum_{k=1}^{n_{el}}(\mu_u B_{2k} + \mid \sigma \mid B_{1k}+ B_{3k})) \|R_1^{h,n,\theta}\| \hspace{30 mm} \\
& \leq C_{\tau_1} (\sum_{k=1}^{n_{el}} \bar{B}_{1k}) \|R_1^{h,n,\theta}\|  \\
\sum_{k=1}^{n_{el}} \tau_1'(R_2^{h,n,\theta}, \mu(c) \Delta I^h_{u2} e_{u2}-\sigma I^h_{u2} e_{u2}+ \frac{\partial I^h_{p} e_{p}}{\partial y})_k 
& \leq \mid \tau_1'\mid (\sum_{k=1}^{n_{el}}(\mu_u B_{2k}' + \mid \sigma \mid B_{1k}'+ B_{3k}')) \|R_2^{h,n,\theta}\| \\
& \leq C_{\tau_1} (\sum_{k=1}^{n_{el}} \bar{B}_{2k}) \|R_2^{h,n,\theta}\| \hspace{2mm}(say)\\
\end{split}
\end{equation}
where $\bar{B}_{1k}$= $(\mu_u B_{2k} + \mid \sigma \mid B_{1k}+ B_{3k})$ and $\bar{B}_{2k}$= $(\mu_u B_{2k}' + \mid \sigma \mid B_{1k}'+ B_{3k}') $
and
\begin{multline}
\sum_{k=1}^{n_{el}} \tau_3'(R_4^{h,n,\theta}+d_4, \bigtriangledown \cdot \tilde{\bigtriangledown} I^h_c e_c + \textbf{u} \cdot \bigtriangledown I^h_c e_c - \alpha I^h_c e_c)_k \\
=\sum_{k=1}^{n_{el}} \tau_3'(R_4^{h,n,\theta},\bigtriangledown \cdot \tilde{\bigtriangledown} I^h_c e_c + \textbf{u} \cdot \bigtriangledown I^h_c e_c - \alpha I^h_c e_c)_k) + \sum_{k=1}^{n_{el}} \tau_3'(d_4, \bigtriangledown \cdot \tilde{\bigtriangledown} I^h_c e_c + \textbf{u} \cdot \bigtriangledown I^h_c e_c - \alpha I^h_c e_c)_k \\
\leq (\mid \tau_3'\mid + \frac{\tau_3'^2}{dt-\mid \tau_3'\mid})(\sum_{k=1}^{n_{el}}(D_{1m} B_{7k}+D_{2m} B_{7k}' + D_{u1} B_{8k}+D_{u2} B_{8k}' + \mid \alpha \mid B_{6k} )) \|R_4^{h,n,\theta}\| \hspace{21mm} \\
\leq \frac{C_{\tau_3} T_0}{T_0-C_{\tau_3}}(\sum_{k=1}^{n_{el}} \bar{B}_{4k}) \|R_4^{h,n,\theta}\| \hspace{2mm} (say) \hspace{60mm}
\end{multline}
where $\bar{B}_{4k}$= $(D_{1m} B_{7k}+D_{2m} B_{7k}' + D_{u1} B_{8k}+D_{u2} B_{8k}' + \mid \alpha \mid B_{6k} )$.\\
Finally
\begin{equation} 
\begin{split}
\sum_{k=1}^{n_{el}} \tau_2'(R_3^{h,n,\theta}, \bigtriangledown \cdot I^h_\textbf{u} e_\textbf{u})_k & \leq C_{\tau_2}(\sum_{k=1}^{n_{el}} (B_{4k}+B_{5k})) \|R_3^{h,n,\theta}\| \\
\end{split}
\end{equation}
Now this completes finding bounds for each term in the $RHS$ of (84). Therefore our next work is to combine all the results into equation (87). Putting common terms all together to the left hand side and then taking summation over $n=0,...,(N-1)$ on both sides of (87) and multiplying them by $2dt$ we will finally have
\begin{multline}
\|e_c^N\|^2+(2\mu_l-\epsilon_1) \sum_{n=0}^{N-1} \{ \sum_{i=1}^2(\|e_{ui}^{n,\theta}\|^2+ \|\frac{\partial e_{ui}^{n,\theta}}{\partial x}\|^2 + \|\frac{\partial e_{ui}^{n,\theta}}{\partial y}\|^2)+ \\
(2\sigma-\epsilon_1) \|e_p^{n,\theta}\|^2 +(2 D_{\alpha}-\epsilon_1)(\|e_{c}^{n,\theta}\|^2+
 \|\frac{\partial e_{c}^{n,\theta}}{\partial x}\|^2 + \|\frac{\partial e_{c}^{n,\theta}}{\partial y}\|^2)\} dt\\
 \leq \frac{h^2}{ \epsilon_1} \sum_{n=0}^{N-1} (\|R_1^{h,n,\theta}\|^2+\|R_2^{h,n,\theta}\|^2+\|R_3^{h,n,\theta}\|^2+\|R_4^{h,n,\theta}\|^2)dt + \hspace{17 mm}\\
 2C_{\tau_1} \{(\sum_{k=1}^{n_{el}} \bar{B}_{1k}) \sum_{n=0}^{N-1} \|R_1^{h,n,\theta}\| +  (\sum_{k=1}^{n_{el}} \bar{B}_{2k}) \sum_{n=0}^{N-1} \|R_2^{h,n,\theta}\|\} dt +2 C_{\tau_2} (\sum_{k=1}^{n_{el}} (B_{4k} \\
 +B_{5k})) \sum_{n=0}^{N-1} \|R_3^{h,n,\theta}\| dt + \frac{4 C_{\tau_3}}{T_0-C_{\tau_3}} (\sum_{k=1}^{n_{el}}(B_{6k}+\bar{B}_{4k}) \sum_{n=0}^{N-1} \|R_4^{h,n,\theta}\| dt \\
 \leq \bar{C}(\textbf{R}^h)(h^2+ dt^{2r}) \hspace{75 mm}
\end{multline}
Now taking minimum of the coefficients in the left hand side and dividing both sides by them we will have aposteriori estimate, which does not depend upon exact solution. It shows that the method is second order accurate in space.

\section{Numerical Experiment}
It is well-known that Galerkin finite element method suffers from lack of stability for small diffusion. Consequently stabilized finite element methods enter into studies. In this section few numerical experiments have been carried out to verify theoretically established rate of convergence for algebraic subgrid scale(ASGS) stabilized finite element method. Here we have taken two cases into account. First case presents comparison of two methods Galerkin and ASGS for non-zero diffusion, whereas the second case does the same for zero diffusion. \vspace{1mm}\\
For simplicity we have considered bounded square domain $\Omega$= (0,1) $\times$ (0,1) where the  fluid flow phenomena is governed by unified Stokes-Darcy-Brinkman equation and the dispersion of the solute is modelled by VADR equation. The expression of concentration dependent viscosity is taken from \cite{RefR}, which establishes that viscosity of a solvent depends upon concentration of the solute of a electrolyte solution. Let us mention here the exact solutions as follows: \vspace{1mm}\\
$\textbf{u}=(t sin^2(\pi x) sin(\pi y) cos(\pi y), -t sin(\pi x) cos(\pi x) sin^2(\pi y))$, \vspace{1mm} \\
 $p=t sin(2 \pi x) cos(2 \pi y)$ and $c= t x y (x-1)(y-1)$ \vspace{1mm} \\
 In both cases the viscosity \cite{RefR} is taken to be, $\mu(c)=0.954 e^{27.93 \times 0.028 c}$ \vspace{1mm}\\
Figure 1 shows mesh plot for 40 $\times$ 40 grid points and figure 2,3,4 show horizontal velocity plot, velocity plot and velocity concentration plot respectively for those grid points under ASGS method. It is sufficient to show these plot only for ASGS since same type of plots are generated under Galerkin method.\vspace{1mm}\\
Again figure 8 and 9 present surface plot of exact solution and ASGS solution respectively at 40 $\times$ 40 grid points. These plots are clearly showing that for finer grid the Subgrid solution is more accurate with respect to exact one. \vspace{2mm}\\
\textbf{First case: Non-zero diffusion coefficients} 
The diffusion coefficients and stabilization parameters are considered as follows: \vspace{1mm}\\
$D_1=t^2(sin(\pi x))^4(sin(2 \pi y))^2$, $D_2=t^2 (sin(2 \pi x))^2 (sin(\pi y))^4$ \vspace{1mm}\\
$\tau_1=(4 \frac{\mu_l}{h^2}+ \sigma)^{-1}$, $\tau_2=(4 \sigma h + 0.001 \mu_l)$ and $\tau_3=(\frac{9}{4 h^2}+ \frac{3}{2h} + \alpha )^{-1}$ \vspace{1mm}\\
where $\mu_l= 0.954 e^{27.93 \times 0.028 \times 0.0625}$, $\sigma=1$, $\alpha=0.01$ \vspace{1mm}\\
Table 1 and table 2 present the error in $H^1$ norm and order of convergence under  Galerkin method and SGS method respectively for this case. These tables are clearly showing that both the methods perform equally well for non-zero diffusion coefficients. \vspace{1mm}\\
Figure 5 represents the comparison of exact solution with both Galerkin and ASGS solutions for non-zero diffusion coefficients with respect to increasing time.\vspace{2mm}\\
\textbf{Second case: Zero diffusion coefficients}
In this case both the diffusion coefficients are taken to be zero. Hence the stabilization parameters for ASGS method are turned out to be as follows: \vspace{1mm}\\
$\tau_1=(4 \frac{\mu_l}{h^2}+ \sigma)^{-1}$, $\tau_2=(4 \sigma h + 0.001 \mu_l)$ and $\tau_3=( \frac{3}{2h} + \alpha )^{-1}$ \vspace{1mm} \\
The coefficients are taken the same values as above case. \vspace{1mm}\\
Table 3 and table 4 present the error in $H^1$ norm and order of convergence under Galerkin method and ASGS method respectively. The tables represent that for zero diffusion the order of convergence under Galerkin method oscillates, whereas ASGS performs well. \vspace{1mm}\\
Figure 6 shows the comparison of exact solution with both Galerkin and ASGS solutions for zero diffusion coefficients with respect to increasing time.

\begin{table}[]
    \centering
    \begin{tabular}{||c c c||}
    \hline 
    Mesh size & Error in $H^1$ norm &  Order of convergence \\
    \hline \hline
      10      &   0.953771        &                       \\
      \hline
      20      &  0.275123  &  1.79357 \\
      \hline
      40      &  0.0635331   & 2.1145  \\
      \hline
      80      &  0.0190837  & 1.73516  \\
      \hline
      160     &  0.00531838   & 1.84328  \\
      \hline
    \end{tabular}
\caption{Error and Order of convergence obtained in $H^1$ norm under Galerkin method for non-zero diffusion coefficient}
\end{table}

\begin{table}[]
    \centering
    \begin{tabular}{||c c c||}
    \hline 
    Mesh size & Error in $H^1$ norm &  Order of convergence \\
    \hline \hline
      10      &   0.19916        &                       \\
      \hline
      20      &  0.0670151  &  1.57137 \\
      \hline
      40      &  0.0167876   &  1.99709 \\
      \hline
      80      &  0.00461748  & 1.86221 \\
      \hline
      160     &  0.0099767   & 2.21047  \\
      \hline
    \end{tabular}
\caption{Error and Order of convergence obtained in $H^1$ norm under ASGS method for non-zero diffusion coefficient}
\end{table}

\begin{table}[]
    \centering
    \begin{tabular}{||c c c||}
    \hline 
    Mesh size & Error in $H^1$ norm &  Order of convergence \\
    \hline \hline
      10      &   0.953883        &                       \\
      \hline
      20      &  0.27514   & 1.79364 \\
      \hline
      40      &  0.0635417   & 2.11439  \\
      \hline
      80      &  0.019088   & 1.73504  \\
      \hline
      160     &  0.00732083   & 1.38259  \\
      \hline
    \end{tabular}
\caption{Error and Order of convergence obtained in $H^1$ norm under Galerkin method for zero diffusion coefficient}
\end{table}

\begin{table}[]
    \centering
    \begin{tabular}{||c c c||}
    \hline 
    Mesh size & Error in $H^1$ norm &  Order of convergence \\
    \hline \hline
      10      &   0.200471        &                       \\
      \hline
      20      &  0.067453  &  1.57144 \\
      \hline
      40      &  0.0169767   & 1.99033  \\
      \hline
      80      &  0.0046746   & 1.86064  \\
      \hline
      160     &  0.000980474   & 2.23555  \\
      \hline
    \end{tabular}
\caption{Error and Order of convergence obtained in $H^1$ norm under ASGS method for zero diffusion coefficient}
\end{table}

\begin{figure}
\centering
\begin{minipage}{.4\textwidth}
  \centering
  \includegraphics[width=\textwidth]{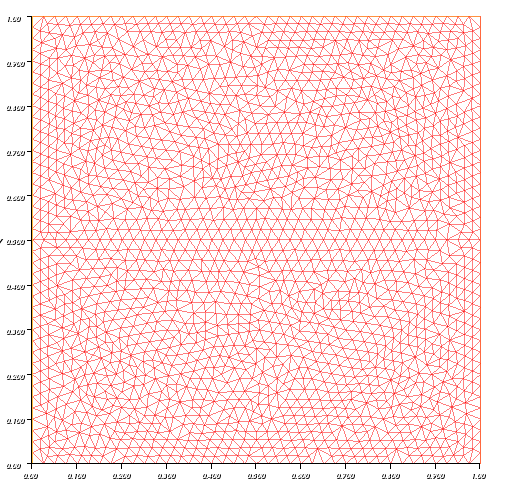} 
  \caption{Mesh for 40 $\times$ 40 grid points}
\end{minipage}
\begin{minipage}{.4\textwidth}
  \centering
  \includegraphics[width=\textwidth]{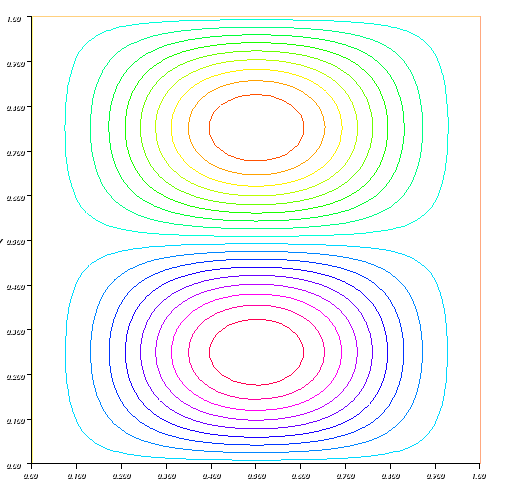}  
  \caption{Horizontal velocity plot}
\end{minipage}
\end{figure}

\begin{figure}
\centering
\begin{minipage}{.4\textwidth}
  \centering
  \includegraphics[width=\textwidth]{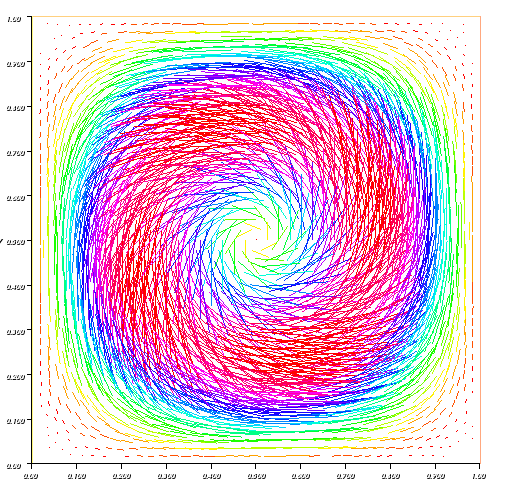} 
  \caption{Velocity plot}
\end{minipage}
\begin{minipage}{.4\textwidth}
  \centering
  \includegraphics[width=\textwidth]{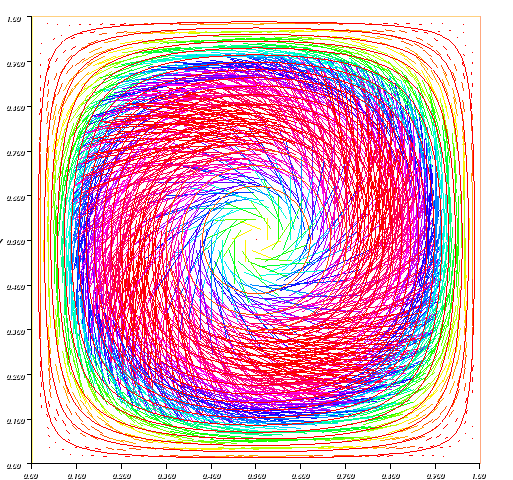}  
  \caption{Velocity concentration plot}
\end{minipage}
\end{figure}

\begin{figure}
\centering
\begin{minipage}{.8\textwidth}
\centering
\includegraphics[width=\textwidth]{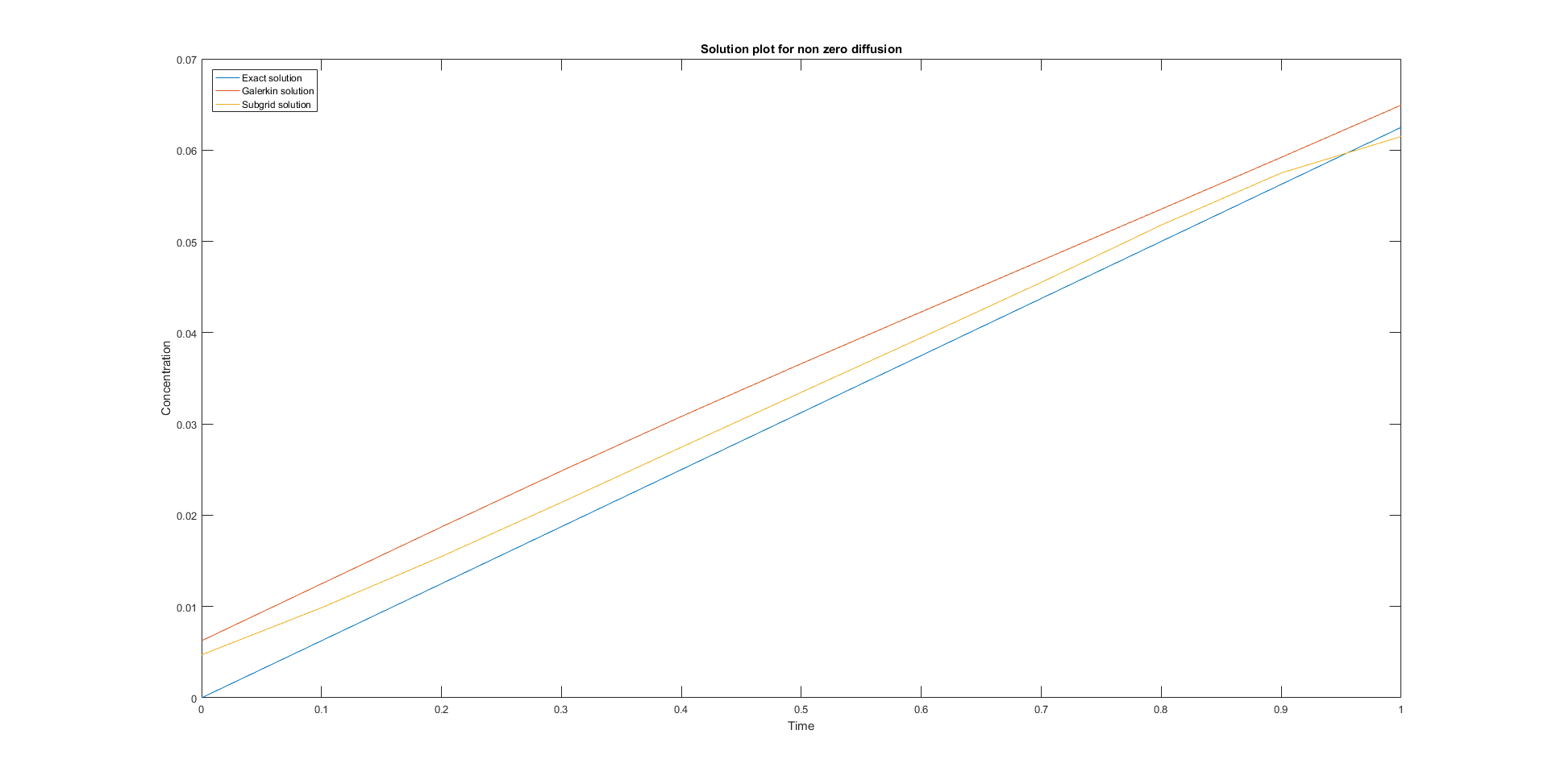}
\caption{Comparison of Exact, Galerkin and ASGS solutions for non-zero diffusion coefficient}
\end{minipage}
\end{figure}

\begin{figure}
\centering
\begin{minipage}{.8\textwidth}
\centering
\includegraphics[width=\textwidth]{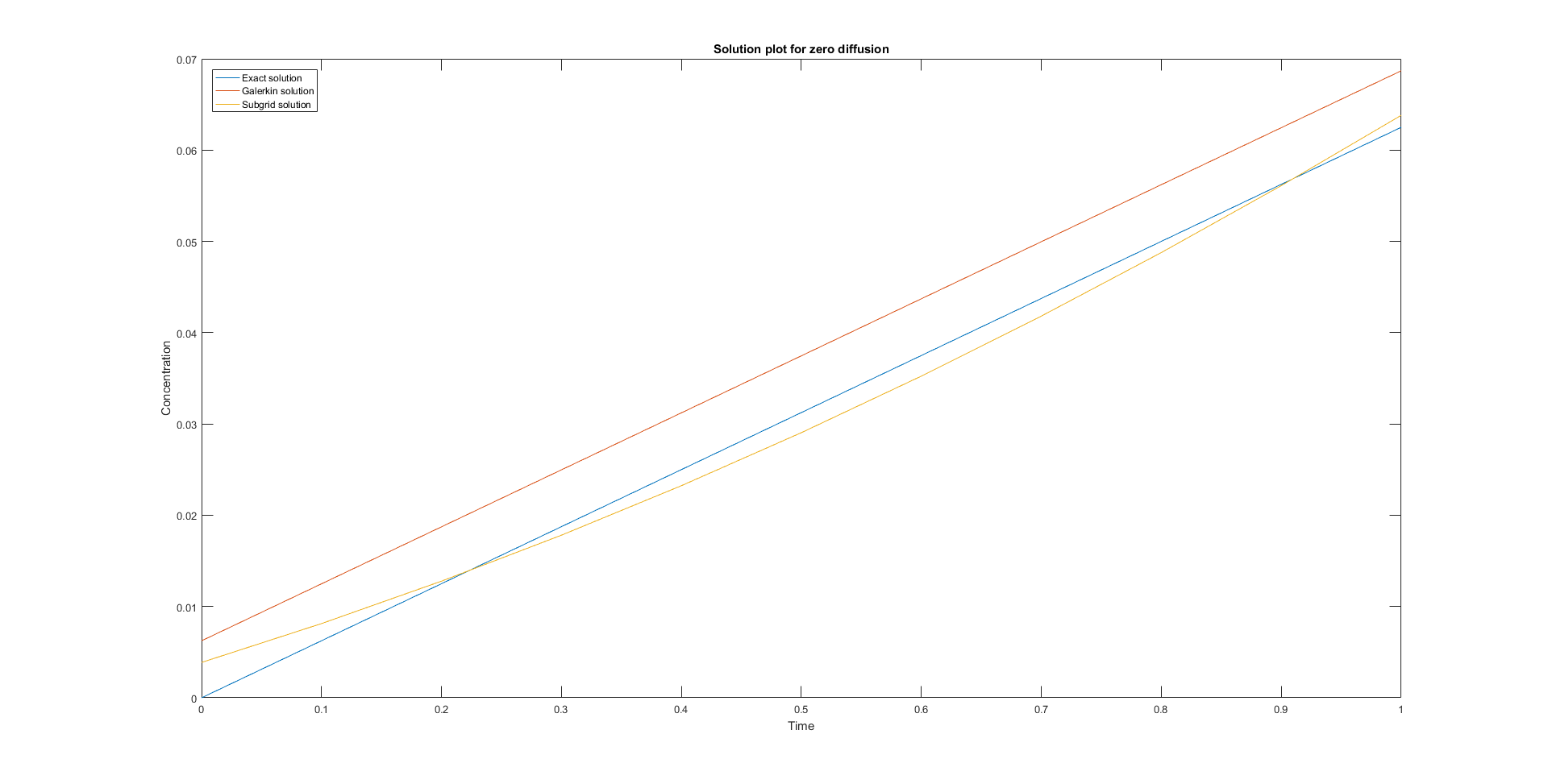}
\caption{Comparison of Exact, Galerkin and ASGS solutions for zero diffusion coefficient}
\end{minipage}
\end{figure}

\begin{figure}
\centering
\begin{minipage}{.8\textwidth}
\centering
\includegraphics[width=\textwidth]{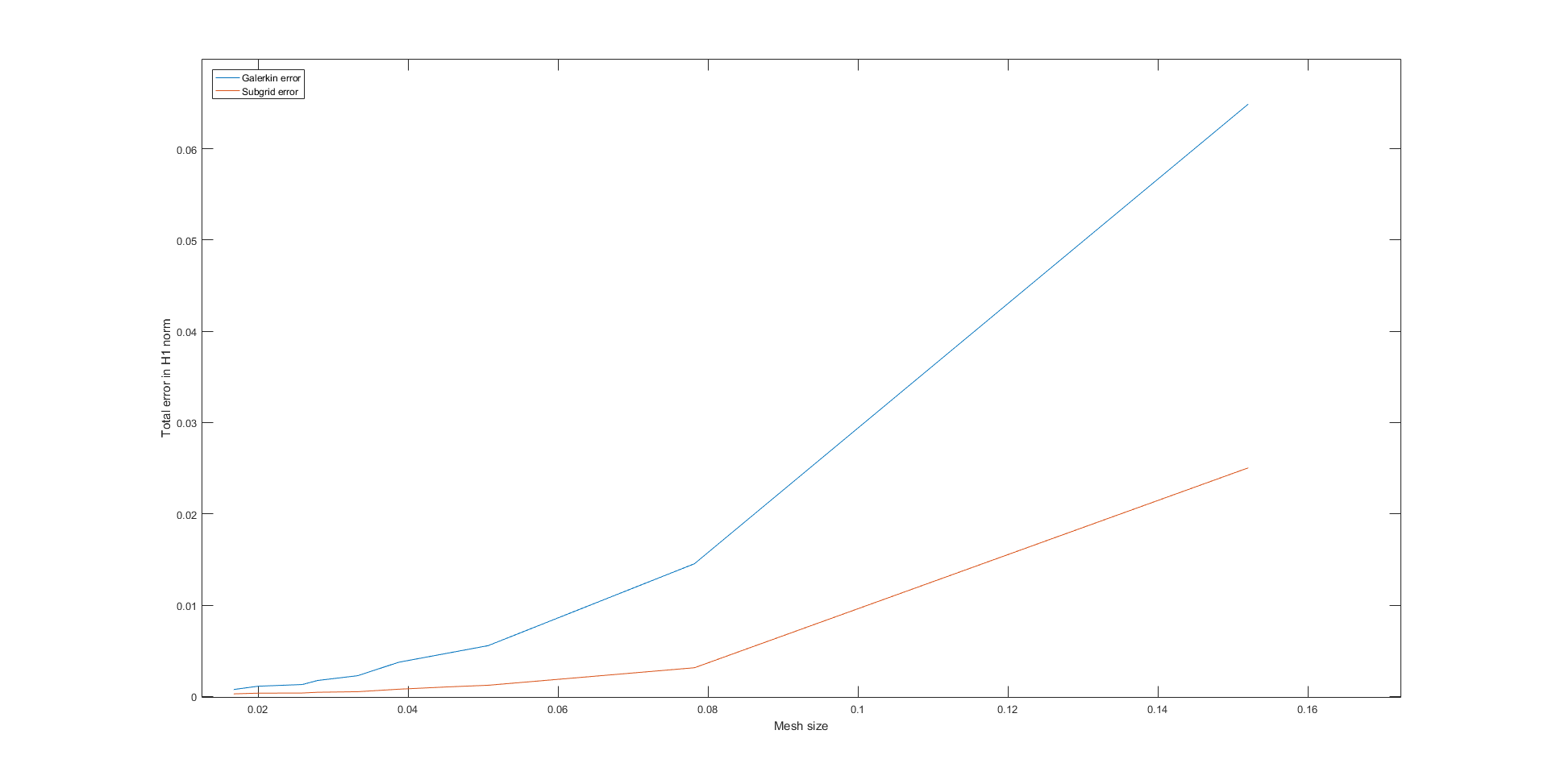}
\caption{Error plot in $H^1$ norm for Galerkin and ASGS method}
\end{minipage}
\end{figure}

\begin{figure}
\centering
\begin{minipage}{.4\textwidth}
  \centering
  \includegraphics[width=\textwidth]{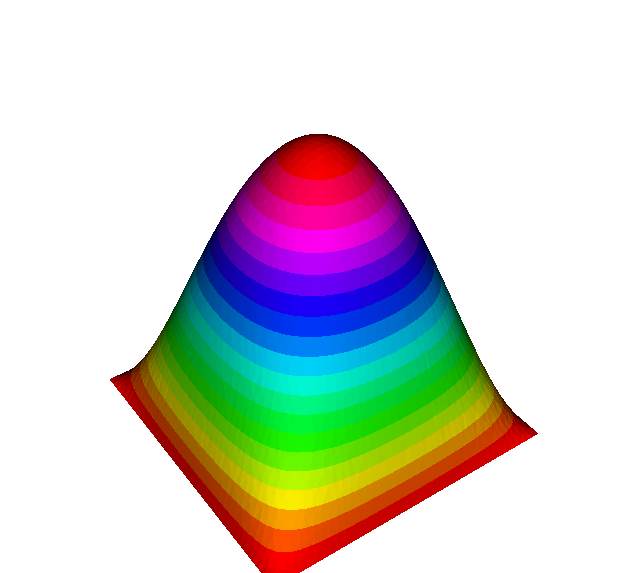} 
  \caption{Exact solution for 40 $\times$ 40 grid points}
\end{minipage}
\begin{minipage}{.4\textwidth}
  \centering
  \includegraphics[width=\textwidth]{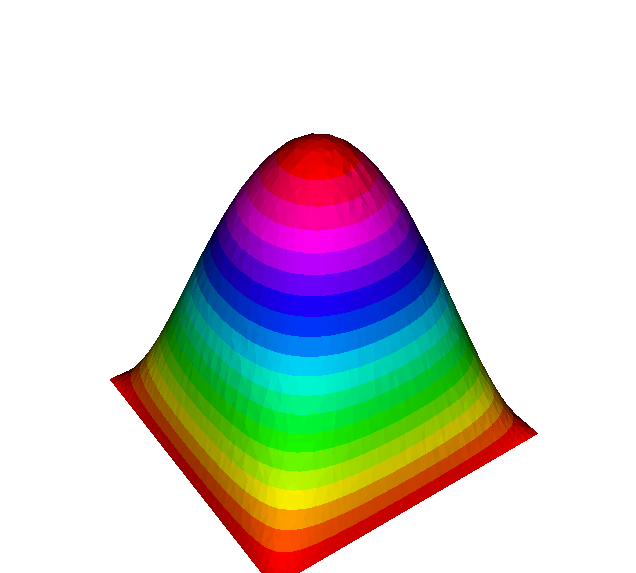}  
  \caption{ASGS solution for 40 $\times$ 40 grid points}
\end{minipage}
\end{figure}

\begin{remark}
The tables are showing that error under ASGS method at each mesh size is turned out to be lesser than that of Galerkin method and for both cases the order of convergence under ASGS method is 2, which justifies theoretically established result.
\end{remark}

\begin{remark}
It is clear from comparison of figure 5 and figure 6 that ASGS solution comparatively more fast converges to exact solution whereas the Galerkin solution is converging to exact solution only for the case of non-zero diffusion coefficients.
\end{remark}

\begin{remark}
Figure 7 represents the error plot in $H^1$ norm under Galerkin and ASGS method. It shows that error under ASGS method is much lesser than that of Galerkin method at the same mesh size and both are decreasing for finer mesh.
\end{remark}

\section{Conclusion}
Apriori and aposteriori error estimations have been carried out for stabilized ASGS finite element method employed on strongly coupled unified Stokes-Darcy-Brinkman/VADR Transport model. Theoretically established results are verified well by numerical experiments. Better performance of ASGS method over Galerkin method is shown in the last section \vspace{1cm}\\
{\large \textbf{Acknowledgement}} \vspace{2mm}\\
This work has been supported by grant from Innovation in Science Pursuit for Inspired Research (INSPIRE) programme sponsored and managed by the Department of Science and Technology(DST), Ministry of Science and Technology, India.

\end{document}